\documentclass{article}
\usepackage{arxiv}
\usepackage{graphicx} 
\usepackage[hidelinks]{hyperref}
\usepackage[maxnames=999,minnames=999]{biblatex}
\usepackage{amsmath}
\usepackage{amsthm}
\usepackage{enumitem}
\usepackage{booktabs}
\usepackage{tikz}
\usepackage{diagbox}

\usepackage{mathtools}

\bibliography{references}
\usepackage[dvipsnames]{xcolor} 
\usepackage{adjustbox}
\usepackage{subcaption}
\usepackage{fontawesome}
\usepackage{float}
\usepackage{appendix}

\usepackage[framemethod=TikZ]{mdframed}

\mdfsetup{
linecolor=black,
linewidth=0.5pt,
backgroundcolor=black!3,
innerbottommargin = 9pt,
roundcorner=5pt}

\usepackage[draft, margin, index]{fixme}
\fxsetup{theme=color, mode=multiuser}
\FXRegisterAuthor{dp}{adp}{\color{purple}Diane} 
\FXRegisterAuthor{db}{adb}{\color{red}Daniel} 
\FXRegisterAuthor{bl}{abl}{\color{blue}Bernard} 
\FXRegisterAuthor{sm}{sml}{\color{green}Sydney}

\usepackage[charter, cal=cmcal]{mathdesign}

\newcommand{\vc}[1]{\ensuremath{\vcenter{\hbox{#1}}}}
\tikzset{flag_pic/.style={scale=1}}  
\tikzset{unlabeled_vertex/.style={inner sep=1.7pt, outer sep=0pt, circle, fill}} 
\tikzset{labeled_vertex/.style={inner sep=2.2pt, outer sep=0pt, rectangle, fill=gray, draw=black}} 
\tikzset{edge_color0/.style={color=black,line width=1.2pt,opacity=0.5,dashed}} 
\tikzset{edge_color1/.style={color=BrickRed,  line width=1.2pt,opacity=1}} 
\tikzset{edge_color2/.style={color=NavyBlue, line width=1.2pt,opacity=1}} 
\tikzset{edge_color3/.style={color=green!80!black,line width=1.2pt}} 
\tikzset{edge_color4/.style={color=orange, line width=1.2pt}} 
\tikzset{edge_color5/.style={color=red,  line width=1.2pt,dotted}} 
\tikzset{edge_color6/.style={color=blue, line width=1.2pt,dotted}} 
\tikzset{edge_color7/.style={color=green, line width=1.2pt,dotted}} 
\tikzset{edge_color8/.style={color=gray, line width=1.2pt}} 
\tikzset{edge_color9/.style={color=gray, dotted, line width=1.2pt}} 
\tikzset{edge_color10/.style={color=gray, dashed, line width=1.2pt}} 
\tikzset{edge_color11/.style={color=pink, dashed, line width=1.2pt}} 
\tikzset{edge_colorroot/.style={color=red, line width=1.7pt}} 
\tikzset{edge_thin/.style={color=black}} 
\tikzset{edge_hidden/.style={color=black,dotted,opacity=0}} 
\tikzset{vertex_color0/.style={inner sep=1.7pt, outer sep=0pt, draw, circle, fill=black}} 
\tikzset{vertex_color1/.style={inner sep=1.7pt, outer sep=0pt, draw, circle, fill=red}} 
\tikzset{vertex_color2/.style={inner sep=1.7pt, outer sep=0pt, draw, circle, fill=blue}} 
\tikzset{vertex_color3/.style={inner sep=1.7pt, outer sep=0pt, draw, circle, fill=green!80!black}} 
\tikzset{vertex_color4/.style={inner sep=1.7pt, outer sep=0pt, draw, circle, fill=pink}} 
\tikzset{vertex_color5/.style={inner sep=1.7pt, outer sep=0pt, draw, circle, fill=pink,label=below:{$5$}}} 
\tikzset{vertex_color6/.style={inner sep=1.7pt, outer sep=0pt, draw, circle, fill=pink,label=below:{$6$}}} 
\tikzset{vertex_color7/.style={inner sep=1.7pt, outer sep=0pt, draw, circle, fill=pink,label=below:{$7$}}} 
\tikzset{vertex_color8/.style={inner sep=1.7pt, outer sep=0pt, draw, circle, fill=pink,label=below:{$8$}}} 
\tikzset{vertex_color9/.style={inner sep=1.7pt, outer sep=0pt, draw, circle, fill=pink,label=below:{$9$}}} 
\tikzset{vertex_color10/.style={inner sep=1.7pt, outer sep=0pt, draw, circle, fill=pink,label=below:{$10$}}} 
\tikzset{vertex_color11/.style={inner sep=1.7pt, outer sep=0pt, draw, circle, fill=pink,label=below:{$11$}}} 
\tikzset{vertex_color12/.style={inner sep=1.7pt, outer sep=0pt, draw, circle, fill=pink,label=below:{$12$}}} 
\tikzset{vertex_color13/.style={inner sep=1.7pt, outer sep=0pt, draw, circle, fill=pink,label=below:{$13$}}} 
\tikzset{vertex_color14/.style={inner sep=1.7pt, outer sep=0pt, draw, circle, fill=pink,label=below:{$14$}}}
\tikzset{labeled_vertex_color0/.style={inner sep=2.2pt, outer sep=0pt, draw, rectangle, fill=black}} 
\tikzset{labeled_vertex_color1/.style={inner sep=2.2pt, outer sep=0pt, draw, rectangle, fill=red}} 
\tikzset{labeled_vertex_color2/.style={inner sep=2.2pt, outer sep=0pt, draw, rectangle, fill=blue}} 
\tikzset{labeled_vertex_color3/.style={inner sep=2.2pt, outer sep=0pt, draw, rectangle, fill=green}} 
\tikzset{labeled_vertex_color4/.style={inner sep=2.2pt, outer sep=0pt, draw, rectangle, fill=pink}} 
\tikzset{text_color0/.style={color=black}} 
\tikzset{text_color1/.style={color=red}} 
\tikzset{text_color2/.style={color=blue}} 
\tikzset{text_color3/.style={color=green!70!black}} 
\tikzset{text_color4/.style={color=orange}} 
\tikzset{text_color5/.style={color=gray}} 

\def\outercycle#1#2{ 
\pgfmathtruncatemacro{\plusone}{#1+1} 
\pgfmathtruncatemacro{\zeroshift}{270 - (#2-1)*360/#1/2 } 
\draw  \foreach \x in {0,1,...,#1}{(\zeroshift+\x*360/#1:1) coordinate(x\x)};}

\def\labelvertex#1{\pgfmathtruncatemacro{\vertexlabel}{#1+1 } \draw (x#1) node{\color{yellow}\tiny\vertexlabel}; }

\usepackage{ifthen}

\tikzset{vertex_u/.style={unlabeled_vertex}} 
\tikzset{vertex_l/.style={labeled_vertex}}  
\tikzset{vertex_su/.style={inner sep=1.2pt, outer sep=0pt, circle, fill}}

\newcommand{\Fuu}[1]{
\,\vc{\begin{tikzpicture}[scale=0.3]\outercycle{2}{1}
\draw[edge_color#1] (x0)--(x1);  
\draw (x0) node[unlabeled_vertex]{};\draw (x1) node[unlabeled_vertex]{};
\end{tikzpicture}}
\,
}

\newcommand{\Fuuu}[3]{
\vc{\begin{tikzpicture}[scale=0.4]\outercycle{3}{1}
\draw[edge_color#1] (x0)--(x1);\draw[edge_color#2] (x0)--(x2);  \draw[edge_color#3] (x1)--(x2);    
\draw (x0) node[unlabeled_vertex]{};\draw (x1) node[unlabeled_vertex]{};\draw (x2) node[unlabeled_vertex]{};
\end{tikzpicture}}}

\newcommand{\FfourEdges}[6]{
\draw[edge_color#1] (x0)--(x1);\draw[edge_color#2] (x0)--(x2);\draw[edge_color#3] (x0)--(x3);  \draw[edge_color#4] (x1)--(x2);\draw[edge_color#5] (x1)--(x3);  \draw[edge_color#6] (x2)--(x3);
}
\newcommand{\Ffour}[5]{
\vc{\begin{tikzpicture}[scale=0.4]\outercycle{4}{2}
\FfourEdges#5
\draw (x0) node[vertex_#1]{};\draw (x1) node[vertex_#2]{};\draw (x2) node[vertex_#3]{};\draw (x3) node[vertex_#4]{};
\ifthenelse{\equal{#1}{l}}{\labelvertex{0}}{}%
\ifthenelse{\equal{#2}{l}}{\labelvertex{1}}{}%
\ifthenelse{\equal{#3}{l}}{\labelvertex{2}}{}%
\ifthenelse{\equal{#4}{l}}{\labelvertex{3}}{}%
\end{tikzpicture}}
}

\newcommand{\Fuuuu}[6]{\Ffour{u}{u}{u}{u}{#1#2#3#4#5#6}}

\usepackage{pgf,tikz,pgfplots}
\usetikzlibrary{shapes.geometric, graphs, fit, positioning}
\tikzset{
    gon1/.style={name=tmp,regular polygon,regular polygon sides=#1,minimum size=10pt,inner sep=0pt},
    gon2/.style={name=tmp,regular polygon,regular polygon sides=#1,minimum size=10pt,inner sep=(1/(2*(tan(pi/#1)))*0.15},
    polygon side/.style args={#1--#2}{
        insert path={(tmp.corner #1)-- (tmp.corner #2)}
    }
}

\newcommand{\FlagGraph}[3][thick]{
\ifnum#2=2%
    \tikz[baseline=0pt]{
        \node[#1, circle,inner sep=1pt,fill] (tmp1) at (0,0){};
        \node[#1,circle,inner sep=1pt,fill] (tmp2) at (0,8pt){};
        \ifx#3\empty%
        \else
            \draw[#1] (0,0) -- (0,8pt);
        \fi
    }
\else%
    \tikz[baseline=-2pt]{
        \node[#1,gon1=#2]{};
        \foreach \X in {1,...,#2}{
            \fill[#1] (tmp.corner \X) circle (1.5pt);
        }
        \draw[#1,polygon side/.list={#3}];
    }
\fi
}

\newcommand{\FlagGraphNumb}[4][very thick]{
\ifnum#2=2%
    \tikz[baseline=3pt]{
        \ifx#3\empty%
        \else
            \draw[#1] (0,0) -- (0,16pt);
        \fi
        \ifnum#4>0
          \node[thick, circle,inner sep=1pt, fill=white, draw] (tmp1) at (0,16pt){\footnotesize 1};
        \else
          \node[#1, circle,inner sep=1.5pt, fill] (tmp1) at (0,16pt){};
        \fi
        \ifnum#4>1
          \node[thick, circle,inner sep=1pt, fill=white, draw] (tmp2) at (0,0){\footnotesize 2};
        \else 
          \node[#1,circle,inner sep=1.5pt, fill] (tmp2) at (0,0){};
        \fi
    }
\else%
    \tikz[baseline=-2pt]{
        \node[#1,gon2=#2]{};
        \draw[#1,polygon side/.list={#3}];
        \foreach \X in {1,...,#2}{
            \pgfmathtruncatemacro\result{\X-1}
            \ifnum#4>\result
              \node [thick, circle,inner sep=1pt, fill=white, draw] at (tmp.corner \X){\footnotesize $\X$};
            \else 
              \node [#1,circle,inner sep=1.5pt, fill] at (tmp.corner \X){};
            \fi
        }

    }
\fi
}

\newcommand{\FlagGraphLabelTwo}[4][very thick]{
  \tikz[baseline=3pt]{
    \ifx#2\empty%
    \else
        \draw[#1] (0,0) -- (0,16pt);
    \fi
    \ifx#30
      \node[#1, circle,inner sep=1.5pt, fill] (tmp1) at (0,16pt){};
    \else
      \node[thick, circle,inner sep=1pt, fill=white, draw, minimum size = 3pt] (tmp1) at (0,16pt){\footnotesize $#3$};
    \fi  
    \ifx#40
      \node[#1,circle,inner sep=1.5pt, fill] (tmp2) at (0,0){};
    \else
      \node[thick, circle,inner sep=1pt, fill=white, draw, minimum size = 3pt] (tmp2) at (0,0){\footnotesize $#4$}; 
    \fi    
  }
}

\newcommand{\FlagGraphLabelThree}[5][very thick]{
  \tikz[baseline=-2pt]{
        \node[#1,gon2=3]{};
        \draw[#1,polygon side/.list={#2}];
        
        \ifx#30
          \node [#1,circle,inner sep=1.5pt, fill] at (tmp.corner 1){};
        \else 
          \node [thick, circle,inner sep=1pt, fill=white, draw, minimum size = 3pt] at (tmp.corner 1){\footnotesize $#3$};
        \fi
        \ifx#40
          \node [#1,circle,inner sep=1.5pt, fill] at (tmp.corner 2){};
        \else 
          \node [thick, circle,inner sep=1pt, fill=white, draw, minimum size = 3pt] at (tmp.corner 2){\footnotesize $#4$};
        \fi
        \ifx#50
          \node [#1,circle,inner sep=1.5pt, fill] at (tmp.corner 3){};
        \else 
          \node [thick, circle,inner sep=1pt, fill=white, draw, minimum size = 3pt] at (tmp.corner 3){\footnotesize $#5$};
        \fi
    }
}

\newcommand{\FlagGraphLabelFour}[6][very thick]{
  \tikz[baseline=-2pt]{
        \node[#1,gon2=4]{};
        \draw[#1,polygon side/.list={#2}];
        
        \ifx#30
          \node [#1,circle,inner sep=1.5pt, fill] at (tmp.corner 1){};
        \else 
          \node [thick, circle,inner sep=1pt, fill=white, draw, minimum size = 3pt] at (tmp.corner 1){\footnotesize $#3$};
        \fi
        \ifx#40
          \node [#1,circle,inner sep=1.5pt, fill] at (tmp.corner 2){};
        \else 
          \node [thick, circle,inner sep=1pt, fill=white, draw, minimum size = 3pt] at (tmp.corner 2){\footnotesize $#4$};
        \fi
        \ifx#50
          \node [#1,circle,inner sep=1.5pt, fill] at (tmp.corner 3){};
        \else 
          \node [thick, circle,inner sep=1pt, fill=white, draw, minimum size = 3pt] at (tmp.corner 3){\footnotesize $#5$};
        \fi
        \ifx#60
          \node [#1,circle,inner sep=1.5pt, fill] at (tmp.corner 4){};
        \else 
          \node [thick, circle,inner sep=1pt, fill=white, draw, minimum size = 3pt] at (tmp.corner 4){\footnotesize $#6$};
        \fi
    }
}

\newcommand{\FlagGraphLabelFive}[7][very thick]{
  \tikz[baseline=-2pt]{
        \node[#1,gon2=5]{};
        \draw[#1,polygon side/.list={#2}];
        
        \ifx#30
          \node [#1,circle,inner sep=1.5pt, fill] at (tmp.corner 1){};
        \else 
          \node [thick, circle,inner sep=1pt, fill=white, draw, minimum size = 3pt] at (tmp.corner 1){\footnotesize $#3$};
        \fi
        \ifx#40
          \node [#1,circle,inner sep=1.5pt, fill] at (tmp.corner 2){};
        \else 
          \node [thick, circle,inner sep=1pt, fill=white, draw, minimum size = 3pt] at (tmp.corner 2){\footnotesize $#4$};
        \fi
        \ifx#50
          \node [#1,circle,inner sep=1.5pt, fill] at (tmp.corner 3){};
        \else 
          \node [thick, circle,inner sep=1pt, fill=white, draw, minimum size = 3pt] at (tmp.corner 3){\footnotesize $#5$};
        \fi
        \ifx#60
          \node [#1,circle,inner sep=1.5pt, fill] at (tmp.corner 4){};
        \else 
          \node [thick, circle,inner sep=1pt, fill=white, draw, minimum size = 3pt] at (tmp.corner 4){\footnotesize $#6$};
        \fi
        \ifx#70
          \node [#1,circle,inner sep=1.5pt, fill] at (tmp.corner 5){};
        \else 
          \node [thick, circle,inner sep=1pt, fill=white, draw, minimum size = 3pt] at (tmp.corner 5){\footnotesize $#7$};
        \fi
    }
}

\newcommand{\FlagGraphLabelSix}[8][very thick]{
  \tikz[baseline=-2pt]{
        \node[#1,gon2=6]{};
        \draw[#1,polygon side/.list={#2}];
        
        \ifx#30
          \node [#1,circle,inner sep=1.5pt, fill] at (tmp.corner 1){};
        \else 
          \node [thick, circle,inner sep=1pt, fill=white, draw, minimum size = 3pt] at (tmp.corner 1){\footnotesize $#3$};
        \fi
        \ifx#40
          \node [#1,circle,inner sep=1.5pt, fill] at (tmp.corner 2){};
        \else 
          \node [thick, circle,inner sep=1pt, fill=white, draw, minimum size = 3pt] at (tmp.corner 2){\footnotesize $#4$};
        \fi
        \ifx#50
          \node [#1,circle,inner sep=1.5pt, fill] at (tmp.corner 3){};
        \else 
          \node [thick, circle,inner sep=1pt, fill=white, draw, minimum size = 3pt] at (tmp.corner 3){\footnotesize $#5$};
        \fi
        \ifx#60
          \node [#1,circle,inner sep=1.5pt, fill] at (tmp.corner 4){};
        \else 
          \node [thick, circle,inner sep=1pt, fill=white, draw, minimum size = 3pt] at (tmp.corner 4){\footnotesize $#6$};
        \fi
        \ifx#70
          \node [#1,circle,inner sep=1.5pt, fill] at (tmp.corner 5){};
        \else 
          \node [thick, circle,inner sep=1pt, fill=white, draw, minimum size = 3pt] at (tmp.corner 5){\footnotesize $#7$};
        \fi
        \ifx#80
          \node [#1,circle,inner sep=1.5pt, fill] at (tmp.corner 6){};
        \else 
          \node [thick, circle,inner sep=1pt, fill=white, draw, minimum size = 3pt] at (tmp.corner 5){\footnotesize $#8$};
        \fi
    }
}

\newcommand{\FlagGraphLabelSeven}[9][very thick]{
  \tikz[baseline=-2pt]{
        \node[#1,gon2=7]{};
        \draw[#1,polygon side/.list={#2}];
        
        \ifx#30
          \node [#1,circle,inner sep=1.5pt, fill] at (tmp.corner 1){};
        \else 
          \node [thick, circle,inner sep=1pt, fill=white, draw, minimum size = 3pt] at (tmp.corner 1){\footnotesize $#3$};
        \fi
        \ifx#40
          \node [#1,circle,inner sep=1.5pt, fill] at (tmp.corner 2){};
        \else 
          \node [thick, circle,inner sep=1pt, fill=white, draw, minimum size = 3pt] at (tmp.corner 2){\footnotesize $#4$};
        \fi
        \ifx#50
          \node [#1,circle,inner sep=1.5pt, fill] at (tmp.corner 3){};
        \else 
          \node [thick, circle,inner sep=1pt, fill=white, draw, minimum size = 3pt] at (tmp.corner 3){\footnotesize $#5$};
        \fi
        \ifx#60
          \node [#1,circle,inner sep=1.5pt, fill] at (tmp.corner 4){};
        \else 
          \node [thick, circle,inner sep=1pt, fill=white, draw, minimum size = 3pt] at (tmp.corner 4){\footnotesize $#6$};
        \fi
        \ifx#70
          \node [#1,circle,inner sep=1.5pt, fill] at (tmp.corner 5){};
        \else 
          \node [thick, circle,inner sep=1pt, fill=white, draw, minimum size = 3pt] at (tmp.corner 5){\footnotesize $#7$};
        \fi
        \ifx#80
          \node [#1,circle,inner sep=1.5pt, fill] at (tmp.corner 6){};
        \else 
          \node [thick, circle,inner sep=1pt, fill=white, draw, minimum size = 3pt] at (tmp.corner 5){\footnotesize $#8$};
        \fi
        \ifx#90
          \node [#1,circle,inner sep=1.5pt, fill] at (tmp.corner 7){};
        \else 
          \node [thick, circle,inner sep=1pt, fill=white, draw, minimum size = 3pt] at (tmp.corner 5){\footnotesize $#9$};
        \fi
    }
}

\newcommand{\Orderedtwo}[1]{
\vc{\begin{tikzpicture}[xscale=0.36,yscale=0.75]
\draw  \foreach \x in {0,1} {(\x, 0)coordinate(x\x)  node[vertex_su]{}};
\ifx#11
\draw[thick] (x0)to[bend left=50](x1);
\fi
\end{tikzpicture}}
}
\newcommand{\Orderedthree}[3]{
\vc{\begin{tikzpicture}[xscale=0.36,yscale=0.75]
\draw  \foreach \x in {0,1,...,2} {(\x, 0)coordinate(x\x)  node[vertex_su]{}};
\ifx#11
\draw[thick] (x0)to[bend left=50](x1);
\fi
\ifx#21
\draw[thick] (x0)to[bend left=50](x2);
\fi
\ifx#31
\draw[thick] (x1)to[bend left=50](x2);
\fi
\end{tikzpicture}}
}

\newcommand{\Orderedfour}[6]{
\vc{\begin{tikzpicture}[xscale=0.36,yscale=0.75]
\draw  \foreach \x in {0,1,...,3} {(\x, 0)coordinate(x\x) node[vertex_su]{}};
\ifx#11
\draw[thick] (x0)to[bend left=50](x1);
\fi
\ifx#21
\draw[thick] (x0)to[bend left=50](x2);
\fi
\ifx#31
\draw[thick] (x0)to[bend left=50](x3);
\fi
\ifx#41
\draw[thick] (x1)to[bend left=50](x2);
\fi
\ifx#51
\draw[thick] (x1)to[bend left=50](x3);
\fi
\ifx#61
\draw[thick] (x2)to[bend left=50](x3);
\fi
\end{tikzpicture}}
}

\usetikzlibrary{calc}
\usetikzlibrary{shadows}

\def\e{1.2}

\tikzset{
insep/.style={inner sep=2.5pt, outer sep=0pt, circle, fill}, 
vtx/.style={inner sep=1.5pt, outer sep=0pt, circle, fill}, 
redthinline/.style={BrickRed,line width=2pt,opacity=1},
bluethinline/.style={NavyBlue,line width=2pt,opacity=1},
redline/.style={BrickRed,line width=0.3cm,opacity=0.5},
blueline/.style={NavyBlue,line width=0.3cm,opacity=0.5},
grayline/.style={gray,line width=2pt},
blackline/.style={gray,line width=2pt},
turanblock/.style={fill=NavyBlue!30},
turan/.style={inner sep=0.12cm, outer sep=0pt, circle, fill=white!30,draw}, 
turangreen/.style={inner sep=0.12cm, outer sep=0pt, circle, fill=green,draw}, 
}

\def\myclip{
\clip (-2.1,-0.5) rectangle (0.85,2.3)  ;
}

\newcommand{\Erdos}{Erd\H{o}s}
\newcommand{\ER}[1]{\mathrm{ER}(#1)}
\newcommand{\CR}[1]{\mathrm{CR}(#1)}
\newcommand{\Ram}[1]{\mathrm{R}(#1)}
\newcommand{\ordRam}[1]{\overrightarrow{\mathrm{R}}(#1)}

\newcommand{\Lidicky}{Lidick\'y}
\newcommand{\Cmono}{\mathcal{C}^{\mathrm{mono}}}
\newcommand{\G}{\mathcal{G}}
\newcommand{\Hcal}{\mathcal{H}}

\newcommand{\Kb}[2]{K_{#1,#2}}
\newcommand{\Kbmin}[2]{K_{#1,#2}^-}
\newcommand{\C}{\mathcal{C}}
\newcommand{\Pmono}[1]{\mathcal{P}^{\mathrm{mono}}_{#1}}
\newcommand{\Palt}[1]{\mathcal{P}^{\mathrm{alt}}_{#1}}
\newcommand{\cmax}{c_{\max}}
\newcommand{\nint}{[n]}
\newcommand{\EKn}{\binom{\nint}{2}}

\newcommand{\edge}[2]{#1#2}

\renewcommand*{\forall}{\text{ for all }}
\newcommand{\forallpwd}{\text{ for all pairwise-distinct }}
\newcommand{\foralld}{\text{ for all distinct }}
\newcommand{\sharpbound}[1]{{\color{blue}\textbf{#1}}}
\newcommand{\sharpboundV}[1]{{\color{blue}\textbf{#1}}} 
\newcommand{\boundV}[1]{{\color{black}{#1}}}

\newcommand{\sharpboundLBUB}[3]{\ensuremath{{\color{blue}\textbf{#1}}#2#3}}

\newcommand{\ramseyboundLBUB}[5]{
\substack{\noindent\begin{minipage}[c][0.45cm][b]{1.25cm}\centering\parbox[c][0.45cm][b]{1.25cm}\centering\scalebox{0.5}[1.0]{#1}\end{minipage}\\\begin{minipage}[c][0.5cm][t]{1.25cm}\centering\parbox[c][0.3cm][c]{1.25cm}{\centering
\ifnum#2=#3
$\sharpboundLBUB{#2}{#4}{#5}$
\else
\ifnum#3=\noexpand\infty
\small$[#2#4, #3)$
\else
\small$[#2#4, #3#5]$
\fi
\fi
}\end{minipage}}
}

\newcommand{\ramseybound}[3]{

\substack{\noindent\begin{minipage}[c][0.45cm][b]{1.25cm}\centering\parbox[c][0.45cm][b]{1.25cm}\centering\scalebox{1.0}[1.0]{#1}\end{minipage}\\\begin{minipage}[c][0.5cm][t]{1.25cm}\centering\parbox[c][0.3cm][c]{1.25cm}{\centering
\ifnum#2=#3
$\sharpbound{#2}$
\else
\ifnum#3=\noexpand\infty
\small$[#2, #3)$
\else
\small$[#2, #3]$
\fi
\fi
}\end{minipage}}
}

\DeclarePairedDelimiter{\abs}{\lvert}{\rvert}
\DeclarePairedDelimiter{\set}\{\}

\newtheorem{theorem}{Theorem}[section]
\newtheorem{definition}{Definition}[section]
\newtheorem{example}{Example}[section]
\newtheorem{claim}{Claim}[section]

\title{Lower and Upper Bounds for Small Canonical and Ordered Ramsey Numbers}
\author{
Daniel Brosch\thanks{Alpen-Adria-Universität Klagenfurt, Universitätsstraße 65–67, 9020 Klagenfurt, Austria, {\tt daniel.brosch@aau.at}.
}
\and
Bernard Lidick\'y\thanks{Department of Mathematics, Iowa State University, Ames, IA, E-mail: {\tt lidicky@iastate.edu}. Research of this author is supported in part by NSF FRG DMS-2152490 and Scott Hanna Professorship.}
\and
Sydney Miyasaki\thanks{Department of Mathematics, Iowa State University, Ames, IA, E-mail: {\tt miyasaki@iastate.edu}. Research of this author is supported in part by NSF FRG DMS-2152490.}
\and
Diane Puges\thanks{Alpen-Adria-Universität Klagenfurt, Universitätsstraße 65–67, 9020 Klagenfurt, Austria, {\tt diane.puges@aau.at}.
This research was funded in whole or in part by the Austrian Science Fund (FWF) [10.55776/DOC78]. For open access purposes, the author has applied a CC BY public copyright license to any author-accepted manuscript version arising from this submission.
}}
\date{\today}

\pgfplotsset{compat=1.18} 

\usepackage{fontawesome}

\begin{document}

\newcommand{\LBSA}{\ensuremath{_{\text{\scriptsize\faFire}}}}
\newcommand{\UBFA}{\ensuremath{^{\text{\scriptsize\faFlag}}}}

\maketitle

\begin{abstract}
Ramsey numbers form a classical topic in combinatorics and have been extended from graphs to a range of other combinatorial settings. In this paper, we investigate three such extensions.

We first consider \emph{ordered Ramsey numbers}, which can be viewed as being as old as classical Ramsey numbers, since the Erd\H{o}s--Szekeres lemma already fits within this framework. Here, we ask for a monochromatic copy of a linearly ordered graph~$G$ in every $2$-edge-coloring of a linearly ordered complete graph~$K_n$. The smallest such~$n$ is denoted by~$\vec{R}(G)$.

Next, we study \emph{canonical Ramsey numbers}. A \emph{canonical coloring} of a linearly ordered graph~$G$ is an edge-coloring in which~$G$ is monochromatic, rainbow, or min/max-lexicographic. In the latter case, each pair of edges receives the same color if and only if they share the same first (respectively, second) vertex. Erd\H{o}s and Rado showed that for every~$p$ there exists~$n$ such that every edge-coloring of a linearly ordered~$K_n$ contains a canonical copy of~$K_p$; the smallest such~$n$ is denoted by~$ER(G)$.

Finally, we examine \emph{unordered canonical Ramsey numbers}, introduced by Richer. An edge-coloring of~$G$ is \emph{orderable} if there exists a linear ordering of its vertices such that the color of each edge is determined by its first vertex. Unlike lexicographic colorings, this notion also includes monochromatic colorings. Richer proved that for all~$s$ and~$t$, there exists~$n$ such that every edge-coloring of~$K_n$ contains an orderable copy of~$K_s$ or a rainbow~$K_t$. The smallest such~$n$ is denoted by~$CR(s,t)$.

In all three settings, we focus on determining the corresponding Ramsey numbers for small graphs~$G$. We use tabu search and integer programming to obtain lower bounds, and flag algebras or integer programming to establish upper bounds. Among other results, we determine $\vec{R}(G)$ for all graphs~$G$ on up to four vertices except~$K_4^-$, $ER(P_4)$ for all orderings of~$P_4$, and the exact values $CR(6,3)=26$ and $CR(3,5)=13$.

\end{abstract}

\listoffixmes


\section{Introduction}

Ramsey's theorem states that for any two integers $s$ and $t$, there exists a minimal integer $n$ such that every red-blue-coloring of the complete graph of size $n$ contains either a red complete graph of size $s$ or a blue complete graph of size $t$. This integer is called the \emph{Ramsey number} $\Ram{s,t}$.  This result, proved by Frank Ramsey in $1930$, gave rise to an entire area of combinatorics called Ramsey theory.
Ramsey theory is now a very active and prolific area of mathematics which intersects with areas such as graph theory, geometry, number theory, topology, measure theory, among many others.
In this paper, we focus on some variants of the Ramsey numbers that expand the notion to ordered graphs, colorings with an unbounded number of colors, and both. These variants are the ordered and unordered canonical Ramsey numbers and the ordered Ramsey numbers.

\paragraph{Notation.} 
For any graph $G$, we denote by $V(G)$ the set of vertices of $G$ and by $E(G)$ the set of edges of $G$. We denote by $\abs{G} = \abs{V(G)}$ the \emph{size} of $G$. Without loss of generality, when $\abs{G} = n$ we set $V(G) = \nint\coloneqq \{1,\ldots, n\}$. We denote by $K_n = \left\{\nint, \EKn \right\}$ the complete graph of size $n$.

A \emph{linear ordering} (or simply ordering) $\pi$ of $G$ is a permutation of the vertices of $G$. The pair $\G = (G, \pi)$ is called an \emph{ordered graph}. 
 Ordered graphs can be visualized by drawing all the vertices on a horizontal line in order. In Figure~\ref{fig:allP3}, we show all orderings of the path of length $3$ $P_3$.

\begin{figure}[H]
	\captionsetup[subfigure]{font=footnotesize}
	\centering
	\begin{subfigure}{.25\textwidth}
		\centering
            \begin{tikzpicture}[scale=1]
            \draw (0,0) node[vertex_u]{} (1,0) node[vertex_u]{} (2,0) node[vertex_u]{};
            \draw[thick] (0,0) to[bend left=50] (1,0);
            \draw[thick] (1,0) to[bend left=50] (2,0);
            \end{tikzpicture}
            \caption{$(P_3, (1,2,3))$}
            \label{subfig:monP3}
	\end{subfigure}
	\begin{subfigure}{.25\textwidth}
            \centering
            \begin{tikzpicture}[scale=1]
            \draw (0,0) node[vertex_u]{} (1,0) node[vertex_u]{} (2,0) node[vertex_u]{};
            \draw[thick] (0,0) to[bend left=50] (2,0);
            \draw[thick] (1,0) to[bend left=50] (2,0);
            \end{tikzpicture}
            \caption{$(P_3, (1,3,2))$}
            \label{subfig:P3_132}
	\end{subfigure}
	\begin{subfigure}{.25\textwidth}
		\centering
            \begin{tikzpicture}[scale=1]
            \draw (0,0) node[vertex_u]{} (1,0) node[vertex_u]{} (2,0) node[vertex_u]{};
            \draw[thick] (0,0) to[bend left=50] (1,0);
            \draw[thick] (0,0) to[bend left=50] (2,0);
            \end{tikzpicture}
            \caption{$(P_3, (2,1,3))$}
            \label{subfig:P3_213}
	\end{subfigure}
	\caption{All orderings of a path on $3$ vertices.}
	\label{fig:allP3}
\end{figure}
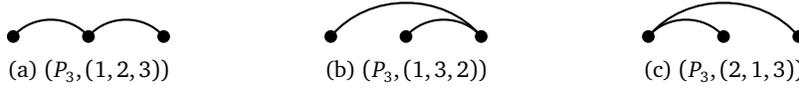

Two ordered graphs $(G, \pi)$ and $(H, \rho)$ are \emph{isomorphic} if $G$ and $H$ are isomorphic via a graph isomorphism that preserves the orderings $\pi$ and $\rho$.
In other words, there exists a bijective mapping $\phi:V(G) \rightarrow V(H)$ such that $\edge \phi(i)\phi(j)$ is an edge of $H$ if and only if $\edge ij$ is an edge of $G$, and $\rho(\phi(i))<\rho(\phi(j))$ if $\pi(i)<\pi(j)$. Simply put, two ordered graphs (on the same vertices) are isomorphic if and only if they are identical. 
$\G = (G, \pi)$ is an \emph{ordered subgraph} of $\Hcal = (H, \rho)$ if $G$ is a subgraph of $H$ and the relative ordering of the vertices of $G$ given by $\pi$ follows the ordering of the vertices in $\Hcal$. 
We also say that $\G$ is \emph{contained} in $\Hcal$, or that $\Hcal$ contains a \emph{copy} of $\G$.

Note that there is only one ordered graph $(K_n, \pi)$ up to isomorphism. To ease notation, we refer to the ordered complete graph of order $n$ simply as $K_n$.

We now provide more details on the three variants of Ramsey numbers considered in this paper. Each variant records the minimum integer 
$n$ for which a certain pattern appears in every coloring of $K_n$  with either two or infinitely many colors. For each of these numbers, we refer to a colored complete graph in which none of these patterns appear as a \emph{Ramsey graph}, and we say that this graph is \emph{feasible}.

\subsection{Canonical Ramsey numbers}

The canonical variant of the Ramsey numbers was introduced in $1950$ by \Erdos{} and Rado. It generalizes Ramsey numbers to colorings with an arbitrary number of colors in ordered graphs. 

\begin{definition}
We say that an ordered graph $\G=(G,\pi)$ is \emph{canonically colored} if one of the following three conditions holds:
\begin{itemize}[label=\raisebox{0.22ex}{\tiny$\bullet$}]
    \item $\G$ is monochromatic: all edges have the same color,
    \item $\G$ is rainbow: all edges have different colors,
    \item $\G$ is (lower/upper) lexicographically colored: edges $\edge{\pi(i)}{\pi(j)}$ and $\edge{\pi(k)}{\pi(\ell)}$ have the same color if and only if $\min(\pi(i),\pi(j)) = \min(\pi(k),\pi(\ell))$ (resp.\ $\max(\pi(i),\pi(j)) = \max(\pi(k),\pi(\ell))$).
\end{itemize}
\end{definition}

In $1950$, \Erdos{} and Rado proved the following theorem, usually called the Canonical Ramsey Theorem.
\begin{theorem}[\Erdos{}, Rado~\cite{Erdoes1950}]
\label{th:ErdosRado}
    For every integer $p$, there exists a least positive integer $n$ such that every edge-coloring of $K_n$ (with an arbitrary number of colors) contains a canonically colored copy of $K_p$.
\end{theorem}

This integer $n$ is called the \emph{canonical Ramsey number}, or \Erdos{}--Rado number, and is denoted by~$\ER{p}$.

\paragraph{Known bounds.} In their proof of Theorem~\ref{th:ErdosRado}, \Erdos{} and Rado provide an upper bound on $\ER{p}$.

This bound was improved by Lefmann and R\H{o}dl~\cite{Lefmann1993,Lefmann1995}, who proved that \begin{equation*}
    2^{c_1p^2} \leq \ER{p} \leq 2^{c_2p^2\log p}
\end{equation*}
for some constants $c_1$ and $c_2$. In $2023$, Kam\v{c}ev and Schacht prove lower bounds for this number in binomial random graphs~\cite{Kamcev2023}.

We also consider the canonical Ramsey number of non-complete graphs. For any ordered graph $\G$, we denote by $\ER{\G}$ the least positive integer $n$ such that every edge-coloring of $K_n$ contains a canonically colored copy of $\G$.

Tighter asymptotic bounds on canonical Ramsey numbers for general graphs were computed for some families of graphs, such as equipartitioned stars~\cite{Jamison2004} and cycles~\cite{Axenovich2005}.
Asymptotic bounds were derived for the asymmetric canonical Ramsey numbers, a variant in which one does not forbid a single canonically colored graph but a family composed of one monochromatic, one rainbow and one lexicographically colored graph~\cite{Jiang2000, Jamison2004, Axenovich2005}. 

Other known bounds for the canonical Ramsey numbers and their variants can be found in Section 3.3. of the dynamic survey on rainbow variants of Ramsey numbers by Fujita, Magnant and Ozeki~\cite{Fujita2010}.

We now examine one of these variants in more depth: the unordered canonical Ramsey numbers.

\subsection{Unordered canonical Ramsey numbers}

In $2000$, Richer~\cite{Richer2000} introduced the \emph{unordered canonical Ramsey numbers}, a variant of the canonical Ramsey numbers in the setting of unordered graphs. To do so, he defined the concept of an \emph{orderable coloring}.
\begin{definition}[Richer~\cite{Richer2000}]
    An edge-coloring  $\chi \ \colon \ E(G)\rightarrow \omega$ of a graph $G$ is called \emph{orderable} if there exists an ordering $\pi$ of the vertices of $G$ such that the color of an edge is completely determined by its lower point. In other words, for any three vertices $i,j,k$ with $i$ adjacent to $j$ and $k$, if $\pi(i)<\pi(j)$ and $\pi(i)<\pi(k)$, then $\chi(\edge{\pi(i)}{\pi(j)}) = \chi(\edge{\pi(i)}{\pi(k)})$.
\end{definition}

Contrary to the definition of a lexicographically colored graph of \Erdos~and Rado, this definition does not require two edges with different lower ends to have different colors. Hence, monochromatic graphs are orderable.

\begin{definition}[Richer~\cite{Richer2000}]
    The unordered canonical Ramsey number $\CR{s,t}$ is the smallest integer $n$, such that every edge-coloring of $K_n$ contains either an orderable copy of $K_s$ or a rainbow copy of $K_t$.
\end{definition}

\paragraph{Known bounds.}
Richer gave bounds on the unordered canonical Ramsey numbers: for every pair $s,t$ of integers, \begin{equation}
\label{eq:boundsUCRN}
    \left(\binom{t}{2}-1\right)^{s-2} + 1 \leq \CR{s,t} \leq 7^{3-s}t^{4s-4}.
\end{equation}

To our knowledge, these remain the best non-asymptotic bounds for general graphs.
Tighter asymptotical bounds were introduced by Babai in $1985$~\cite{Babai1985} and improved by Alon et al. in $1995$~\cite{Lefmann1995}, both papers using probabilistic arguments.
In $2009$, the gap between the asymptotic upper and lower bounds was tightened by Jiang~\cite{Jiang2009}, before being closed by Araujo and Peng in $2024$~\cite{Araujo2024}, who obtained that
\begin{equation*}
    \CR{s,t} = \Theta\left(\frac{c\cdot t^3}{\log t}\right)^{s-2}.
\end{equation*}
 Araujo and Peng introduced as well the variant $\ER{m,l,r}$, which is the minimum $n$ such that every copy of $K_n$ contains either a monochromatic $K_m$, a (strictly) lexicographically colored $K_l$ or a rainbow $K_r$, and studied the asymptotic behavior of this number.

\paragraph{General graphs.}In this paper, we also extend the notion of unordered canonical Ramsey numbers from complete to general graphs.
\begin{definition}
    Let $G$ and $H$ be finite graphs. The unordered canonical Ramsey number $\CR{G, H}$ is the smallest integer $n$ such that every edge-coloring of $K_n$ contains either an orderable copy of $G$ or a rainbow copy of $H$.
\end{definition}

This generalizes the unordered canonical Ramsey number introduced by Richer: in particular, we have $\CR{s,t} = \CR{K_s, K_t}$.

Here, in addition to the usual unordered canonical Ramsey numbers for complete graphs, we will study, in particular, the ones of bipartite graphs. This is partly motivated by the work of Gishboliner, Milojevi\'c, Sudakov and Wigderson~\cite{Gishboliner2024}, who considered an intermediate variant between the ordered and unordered canonical Ramsey Numbers $\ER{H}$ and $\CR{G,H}$, where they do not forbid all orderable colorings, but only the strict lexicographic orderable colorings. They showed that this number grows polynomially for ($d$-regular) bipartite graphs and exponentially for non-bipartite ($d$-regular) graphs.

This gives us reason to think that the unordered canonical Ramsey numbers may behave similarly. The bounds for bipartite graphs may stay within reach more than the ones for complete graphs as the graphs get bigger.

\subsection{Ordered Ramsey Numbers}

We now turn our attention to ordered Ramsey numbers, another variant of Ramsey numbers for ordered graphs. These are more similar to the original Ramsey numbers since they involve $2$-colorings of edges.

\begin{definition}
    The ordered Ramsey number $\ordRam{\G}$ of an ordered graph $\G$ is the smallest integer $n$ such that every $2$-edge-coloring of $K_n$ contains a monochromatic copy of $\G$.
\end{definition}

Here, we restrict ourselves to the diagonal ordered Ramsey numbers. The off-diagonal variant was introduced in~\cite{Choudum2002}: the ordered Ramsey number $\ordRam{\G_1, \ldots, \G_n}$ of $n$ given ordered graphs, is the smallest integer $n$, such that every $n$-edge-coloring $(c_1, \ldots, c_n)$ of $K_n$ contains a copy of $\G_i$ of color $c_i$ for at least one $i$.

 Ordered Ramsey numbers were defined by Choudum and Ponnusamy in $2002$~\cite{Choudum2002}. However, their study took off in $2015$, driven by the simultaneous and independent works of Balko, Cibulka, Kr\'al' and Kyn\v{c}l~\cite{Balko2015, Balko2020} and Conlon, Fox, Lee and Sudakov~\cite{Conlon2017}. They are an interesting extension of Ramsey numbers of graphs because different orderings of the same graph may have different Ramsey numbers.

\paragraph{Known bounds.}
The following are trivial bounds for any ordered graph $\G=(G, \pi)$ on $n$ vertices:
\begin{equation}
    \Ram{G} \leq \ordRam{\G} \leq \Ram{K_n}.
    \label{eq:ORBound1}
\end{equation}

\Erdos{} and Szekeres~\cite{Erdoes1947, Erdoes2009} provided, as soon as $1935$, the following bounds on the Ramsey number of a complete graph\begin{equation*}
     2^{n/2} \leq \Ram{K_n} \leq 2^{2n}.
\end{equation*}
A recent breakthrough by 
 Campos, Griffiths, Morris and  Sahasrabudhe~\cite{Campos2023} improved the upper bound to $\Ram{K_n} \leq 
(4-\varepsilon)^{n}$. Their argument was subsequently optimized by Gupta, Ndiaye, Norin and Wei~\cite{Gupta2024}
who obtained $\Ram{K_n} \leq 
3.8^{n}$.
These results, together with~\eqref{eq:ORBound1}, ensure the existence and finiteness of the Ramsey number for any ordered graph. 
As detailed in~\cite{Conlon2017}, further results of \Erdos{} and Szekeres~\cite{Erdoes2009} can directly be used to deduce that \begin{equation}
    \ordRam{\Pmono{n}} = (n-1)^2-1,
    \label{eq:ordRamMonoPath}
\end{equation} where $\Pmono{n}$ denotes the monotone path on $n$ vertices, as depicted for instance in Figures~\ref{subfig:monP3} and~\ref{subfig:Pmono}.
However, other orderings of paths behave differently.  Let us call \emph{alternating path}, denoted $\Palt{n}$, the ordered path on $n$ vertices $(P_n, \pi_{\mathrm{alt}})$ where $\pi_{\mathrm{alt}} = (1,n, 2, n-1, 3\ldots)$, depicted in figures~\ref{subfig:P3_132} and \ref{subfig:Palt}.  The authors of~\cite{Balko2015, Balko2020} show that the ordered Ramsey numbers of these graphs are linear in $n$ and give their exact values up to $n=8$.
They prove as well that for the monotonous cycle on $n$ vertices $\Cmono_n$ (depicted in~\ref{subfig:monCycle}), we have
\begin{equation*}
    \ordRam{\Cmono_n} = 2n^2 - 6n + 6.
\end{equation*} 
In~\cite{Balko2020} the exact ordered Ramsey numbers of three orderings of $C_4$ are also given.
Choudum and Ponnusamy~\cite{Choudum2002} show that the ordered Ramsey number of any ordering of the star on $n$ vertices is linear in $n$. Overman, Alm, Coffey and Langhoff~\cite{Overman2018} compute upper bounds on several ordered Ramsey numbers of graphs of size $4$, as well as some exact numbers using a SAT solver.

\begin{figure}[H]
	\captionsetup[subfigure]{font=footnotesize}
	\centering
	\begin{subfigure}{.32\textwidth}
            \centering
            \begin{tikzpicture}[scale=.9]
            \foreach \x in {0,1,2,3,4}{\draw (\x,0) node[vertex_u]{};}
            \draw[thick] (0,0) to[bend left=50] (1,0);
            \draw[thick] (1,0) to[bend left=50] (2,0);
            \draw[thick] (2,0) to[bend left=50] (3,0);
            \draw[thick] (3,0) to[bend left=50] (4,0);
            \end{tikzpicture}
            \caption{$\Pmono{5}$ }
            \label{subfig:Pmono}
	\end{subfigure}
	\begin{subfigure}{.32\textwidth}
		\centering
            \begin{tikzpicture}[scale=.9]
            \foreach \x in {0,1,2,3,4}{\draw (\x,0) node[vertex_u]{};}
            \draw[thick] (0,0) to[bend left=50] (4,0);
            \draw[thick] (1,0) to[bend left=50] (4,0);
            \draw[thick] (1,0) to[bend left=50] (3,0);
            \draw[thick] (2,0) to[bend left=50] (3,0);
            \end{tikzpicture}
            \caption{ $\Palt{5}$}
            \label{subfig:Palt}
	\end{subfigure}
        \begin{subfigure}{.32\textwidth}
            \centering
            \begin{tikzpicture}[scale=.9]
            \foreach \x in {0,1,2,3,4}{\draw (\x,0) node[vertex_u]{};}
            \draw[thick] (0,0) to[bend left=50] (1,0);
            \draw[thick] (1,0) to[bend left=50] (2,0);
            \draw[thick] (2,0) to[bend left=50] (3,0);
            \draw[thick] (3,0) to[bend left=50] (4,0);
            \draw[thick] (0,0) to[bend left=50] (4,0);
            \end{tikzpicture}
            \caption{$\Cmono_5$}
            \label{subfig:monCycle}
	\end{subfigure}
	\caption{Monotonous and alternating $P_5$ and monotonous $C_5$.}
	\label{fig:P5C5}
\end{figure}
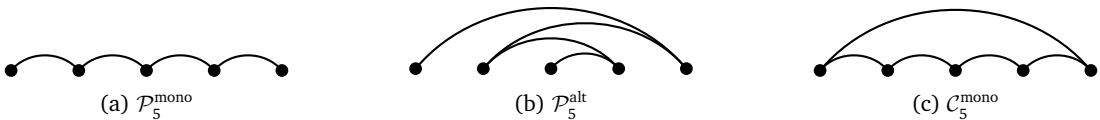

Ordered Ramsey numbers can prove very challenging to compute even for graphs with an ``easy'' structure whose unordered versions have linear Ramsey numbers. For example, it was proven in both~\cite{Balko2020} and~\cite{Conlon2017} that there exist orderings of matchings whose ordered Ramsey number grows super-polynomially. The authors of these papers also derived bounds for graphs depending on their degeneracy, interval chromatic number, or bandwidth. This, in turn, allowed them to show that ordered Ramsey numbers of graphs, for which any of the previous parameters is constant, grow polynomially in the number of vertices of the graph. In~\cite{Neidinger2019}, Neidinger and West obtained new exact ordered Ramsey numbers and some new bounds for several classes of graphs with interval chromatic number equal to $2$. 
In $2019$, Geneson et al.~\cite{Geneson2019} examined the effects of various graph operations on ordered Ramsey numbers, including the addition of a vertex, an edge, or the disjoint union of two graphs. They also derived new upper bounds for several families of matchings.

For a very comprehensive overview of existing bounds and results on ordered Ramsey numbers, we refer to the recent survey of Balko~\cite{Balko2025}.
\subsection{Overview}

We summarize here the notation, types of graphs and forbidden structures of the variants of the Ramsey numbers considered in this paper.
\begin{table}[H]
    \centering\small
    \begin{tabular}{ll|ll}\toprule 
        Variant &  & Graph type & Forbids\\\midrule
        Ramsey number & $\Ram{G}$ & 2-color graph & monochromatic $G$\\[0.5em]
        Canonical Ramsey number & $\ER{G}$ & $\infty$-color ordered graph & \begin{tabular}{@{}l@{}}monochromatic $G$ \\ lexic.\ colored $G$ \\ rainbow $G$ \end{tabular}\\[1.7em]
        Unordered Canonical Ramsey number & $\CR{G,H}$ & $\infty$-color graph & \begin{tabular}{@{}l@{}}orderable $G$ \\ rainbow $H$\end{tabular}\\[0.7em]
        Ordered Ramsey number & $\ordRam{G}$ & 2-color ordered graph & monochromatic $G$\\\bottomrule
    \end{tabular}
    \caption{Overview of the variants of Ramsey numbers we consider.}
    \label{tab:my_label}
\end{table}

\section{Upper bounds: flag algebras}
\label{sec:FA}

Razborov introduced \emph{flag algebras} in the seminal paper~\cite{Razborov2007}, generalizing ideas in extremal combinatorics in a common framework of model theory. He followed with the paper~\cite{razborov_minimal_2008}, where he first applied the theory to determine the edge-triangle graph profile, i.e., the minimum density of triangles given the edge density of a graph. Since then, there have been a multitude of applications of flag algebras to not just extremal graph theory~\cite{MR4585299, MR4472294}, 
but also oriented graphs~\cite{MR4109639,MR4390914},
Ramsey numbers~\cite{2023_KiemPokuttaSpiegel_4colorramsey,Lidicky2021, Lidicky2024},
hypergraphs~\cite{MR4608432,FALGASRAVRY2012}, permutations~\cite{Crudele2024Six, Sliacan2017, Chan2020},
hypercubes~\cite{bodnár2025exactvaluesinducibilitystatistics,Balogh2014Cube}, crossing number~\cite{MR3982073,Balogh2023}, SAT functions~\cite{Balogh2023SAT}, trees~\cite{brosch2024gettingrootproblemsums}, 
vector spaces~\cite{2023_RueSpiegel_Radomultiplicity},
or phylogenetic trees~\cite{Alon2016} to name a few.

While originally designed for asymptotic extremal graph theory, bounding problems in graph limits, their use was first extended to finite graphs by \Lidicky{} and Pfender to compute upper bounds on Ramsey numbers~\cite{Lidicky2021}. Here, we apply their ideas to the colored and ordered variants of Ramsey numbers as introduced above. The approach is analogous to theirs: we compute upper bounds for the edge density in \emph{blow-up graphs}, graph limits obtained from a finite (Ramsey-)graph by replacing each vertex with increasingly large independent sets.

For a detailed and rigorous introduction to flag algebras and their use in computing Ramsey numbers, we refer to \Lidicky{} and Pfender's~\cite{Lidicky2021} work. 
We also recommend \cite{brosch2024gettingrootproblemsums} for an introduction of the method. 
To apply flag algebras to the setting of ordered and colored graphs considered here, we need to work with two variants: the colorblind and the ordered flag algebra, which we explain in this section. 

On an intuitive level, flag algebras work by providing a framework to prove linear relations between densities of small subgraphs. 
For a graph $G$ in $H$, denote by $N(G,H)$ the number of subsets $X$ of size $|V(G)|$ in $H$ such that $X$ induces in $H$ a graph isomorphic to $H$. 
The \emph{density} of $G$ in $H$, denoted by $d(G,H)$, is $N(G,H)$ divided by $\binom{|V(H)|}{|V(G)|}$.
As an example, consider Mantel's theorem claims  that any $n$-vertex triangle-free graph has at most $n^2/4$ edges. 
This can be rephrased as a density question by asking what is maximum $d(K_2,G)$ over all $G$ satisfying $d(K_3,G) = 0$.
Flag algebras answers questions, where the objective and constraints can be express in terms of $d(F,G)$ for $F$s of constant sizes and the size of $G$ going to infinity. 
Since the proof should work for all sequences of $G$, the graph $G$ is often implicit and the problem formulation then reduces to something like maximize $d(K_2)$ subject to $d(K_3)=3$.
Often the $d$ is also dropped and the resulting formulation is
\begin{align*}
    \text{maximize} & \Fuu2\\
    \text{s.t. }& \Fuuu222 = 0.
\end{align*}

The main ingredient in the calculation is establishing that certain linear combination of densities are asymptotically non-negative by writing them as sum-of-squares. 
This allows us to employ semidefinite programming and search for proofs more efficiently and for more complex proofs than humans can.
In the natural setting, flag algebras work for large graphs.
However, investigating small graphs can be also possible. 
We describe the trick in the following section.

\subsection{General method}
\label{ssec:methodFA}

\paragraph{Blow-up graphs.} In order to use methods from asymptotic extremal graph theory, we use the method introduced by \Lidicky{} and Pfender in $2021$ to compute bounds for Ramsey numbers using flag algebras~\cite{Lidicky2021, Lidicky2024}.

A \emph{balanced $k$-blow-up} of a graph $G$ is the graph obtained by replacing every vertex of $G$ with an independent set of size $k$ and every edge by a complete bipartite graph. The blow-up of the cycle graph $C_5$ is depicted in Figure~\ref{fig:C5BU}.
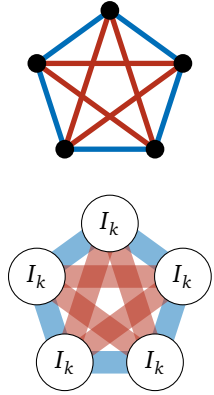
\begin{figure}[H]
\centering
\begin{tikzpicture}
\myclip
\draw[bluethinline] 
(0,0) coordinate (a) --
++(72:\e)coordinate(b)  --
++(2*72:\e)coordinate(c) -- 
++(3*72:\e)coordinate(d) -- 
++(4*72:\e)coordinate(e) -- (a)
 ;
\draw[redthinline](a)--(c);
\draw[redthinline](c)--(e);
\draw[redthinline](e)--(b);
\draw[redthinline](b)--(d);
\draw[redthinline](d)--(a);

\draw
(a) node[insep]{}
(b) node[insep]{}
(c) node[insep]{}
(d) node[insep]{}
(e) node[insep]{}
;
\end{tikzpicture}

\begin{tikzpicture}
\myclip
\draw[blueline] 
(0,0) coordinate (a) --
++(72:\e)coordinate(b)  --
++(2*72:\e)coordinate(c) -- 
++(3*72:\e)coordinate(d) -- 
++(4*72:\e)coordinate(e) -- (a)
 ;
\draw[redline](a)--(c);
\draw[redline](c)--(e);
\draw[redline](e)--(b);
\draw[redline](b)--(d);
\draw[redline](d)--(a);

\draw
(a) node[turan]{$I_k$}
(b) node[turan]{$I_k$}
(c) node[turan]{$I_k$}
(d) node[turan]{$I_k$}
(e) node[turan]{$I_k$}
;
\end{tikzpicture}

\caption{$2$-colored $C_5$ and its even blow-up.}
\label{fig:C5BU}
\end{figure}
Blow-ups of complete (colored) graphs are characterized by forbidding two types of subgraphs on three vertices: graphs with exactly one edge and graphs with exactly two edges with different colors. The forbidden graphs are depicted in Figure \ref{fig:forbidden-blow-up}.

\begin{figure}[H]
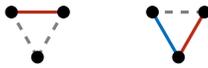

    \centering
    \begin{align*}
         \Fuuu001 \hskip 3em \Fuuu120 
    \end{align*}
    \caption{The forbidden graphs in blow-ups of simple complete 2-edge-colored graphs.}
    \label{fig:forbidden-blow-up}
\end{figure}

For completeness, we give the following claim.
\begin{claim}
    In a blow-up of a complete edge-colored graph, there is no triple of vertices with exactly two non-edges and no triple inducing exactly two edges of different colors.
\end{claim}
\begin{proof}
Let $B(G)$ be a blow-up of the complete edge-colored graph $G$.
Each vertex of $G$ is replaced by an independent set and these are the only maximal independent sets in $B(G)$. 
Observe that if $u$ and $v$ are non-adjacent, then they are part of the same maximal independent set.
Hence, if there are three vertices $u,v,x$ in $V(B(G))$ where both $uv$ and $vx$ are non-edges, then $vx$ must also be a non-edge. This rules out the triple with exactly two non-edges.
In $B(G)$, the complete monochromatic bipartite graph between any two independent sets is monochromatic. If there is a triple $u,v,x \in V(B(G))$ with exactly one non-edge $u,v$, then $u,v$ are in the same independent set, and $x$ is in a different one. Hence, both edges $ux$ and $vx$ come from the same complete monochromatic bipartite graph and have the same color. This rules out the second configuration of a triple with exactly two edges that have different colors.
\end{proof}

If $G$ is a complete graph on $n$ vertices, the $k$-blow up of the graph has exactly $n\binom{k}{2}$ non-edges. Thus, as $k$ grows, the density of non-edges approaches
\begin{equation*}
    \lim_{k\to\infty} \frac{n\binom{k}{2}}{\binom{kn}{2}} = \lim_{k\to\infty}\frac{nk(k-1)/2}{kn(kn-1)/2} = \frac{1}{n}.
\end{equation*}
As the balanced blow-up, under all possible blow-ups of $G$, minimizes the density of non-edges, any lower bound $\delta$ on the asymptotic density of non-edges gives an upper bound of $1/\delta$ on the number of vertices $n$ of $G$. 

\begin{example}
The (usual, two-colored) Ramsey number $R(3,3)$ can be reformulated as
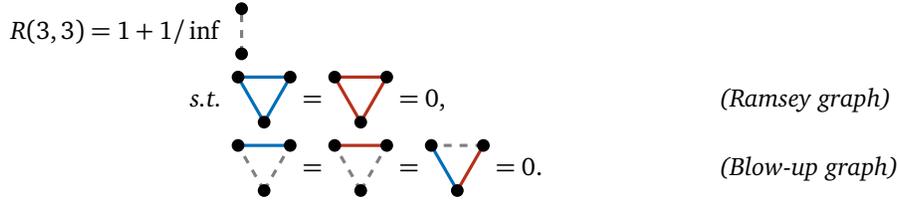
\begin{align*}
    R(3,3) = 1 + 1/\inf \enspace & \Fuu0\\
    \text{s.t. }& \Fuuu222 = \Fuuu111 = 0, && \text{(Ramsey graph)}\\
    & \Fuuu002 = \Fuuu001 = \Fuuu120 = 0. && \text{(Blow-up graph)}\\
\end{align*}

\end{example}

\subsection{Colorblind flag algebras}
\label{ssec:colorblindFA}
All variants of the Ramsey numbers considered here are color-permutation invariant: Both ordered and unordered canonical Ramsey numbers forbid subgraphs solely depending on the \emph{patterns} of colors, not the colors themselves. Of the ordered Ramsey numbers, we only consider diagonal Ramsey numbers, which forbid the same graphs for each color.

Thus, we can state all problems in a \emph{colorblind} fashion: We do not care about the specific colors of the edges, only whether two edges have identical or different colors; we partition the edges into disjoint subsets. While we restrict the number of partitions to two for ordered Ramsey numbers, we allow any number of parts for canonical un-/ordered Ramsey numbers. 

Thus, we can handle infinitely many colors by working with the flag algebra of colorblind graphs. Note that the resulting bounds may worsen compared to working with colored graphs~\cite{Kiem2023}.

\begin{example}
In the colorblind flag algebra (with at most two color classes), the Ramsey number $R(3,3)$ is given by
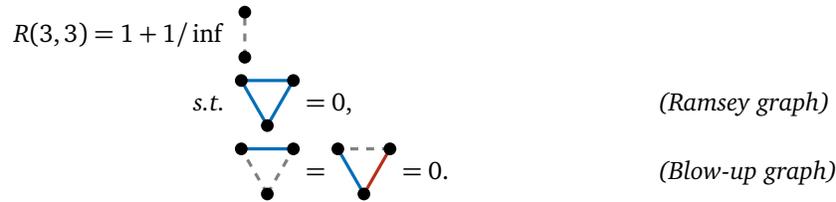
\begin{align*}
    R(3,3) = 1 + 1/\inf \enspace & \Fuu0\\
    \text{s.t. }& \Fuuu222 = 0, && \text{(Ramsey graph)}\\
    & \Fuuu002 = \Fuuu120 = 0. && \text{(Blow-up graph)}\\
\end{align*}
Note that the colors indicate color classes, not specific colors. In particular, setting the blue triangle to zero means all monochromatic triangles are zero.

Similarly, the unordered Ramsey number $\CR{3,4}$ in the colorblind flag algebra (with unlimited color classes) can be reformulated as
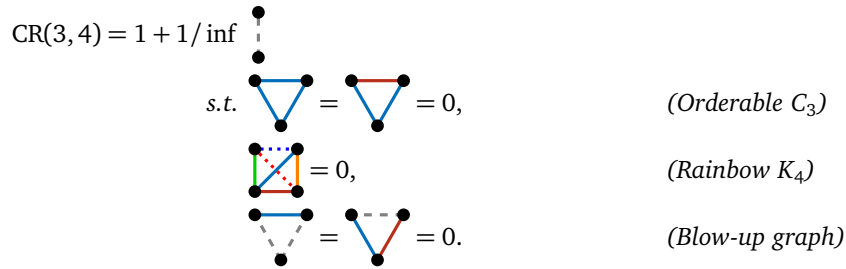
\begin{align*}
    \CR{3,4} = 1 + 1/\inf \enspace & \Fuu0\\
    \text{s.t. }& \Fuuu222 = \Fuuu221 = 0, && \text{(Orderable $C_3$)}\\
    & \Fuuuu123456 = 0, && \text{(Rainbow $K_4$)}\\
    & \Fuuu002 = \Fuuu120 = 0. && \text{(Blow-up graph)}\\
\end{align*}
\end{example}

Formally, we define a colorblind graph as a pair $G=(V, \{E_1, \ldots, E_k\})$, where $V$ is the set of vertices, and the edges $E = \bigcup_i E_i$ are partitioned into disjoint (non-empty) subsets. The parts $E_i$ are unordered; we do not distinguish between two colorblind graphs $(V, \{E_1, E_2\})$ and $(V, \{E_2, E_1\})$. The number of parts $k$ is either at most two (for ordered Ramsey numbers) or any integer (for Canonical Ramsey numbers). For a subset of vertices $S\subseteq V$, a subgraph $G[S]$ of $G$ is the graph $(S, \{E_1 \cap \binom{S}{2},\ldots, E_k\cap \binom{S}{2}\})$ where we remove empty intersections. Two colorblind graphs are isomorphic if and only if there is a graph isomorphism between them that preserves the partition of the edges. 

Thus, colorblind graphs fit into the general theory built by Razborov~\cite{Razborov2007}, and we can define densities, flags, and products in the flag algebra the usual way.

\subsection{Ordered flag algebras}
\label{ssec:orderedFA}

Flag algebra theory is applicable in the setting where the structures are ordered.
A natural ordered structure is permutations. See \cite{Crudele2024Six,Sliacan2017}
for an introduction to flag algebras on permutations.
Permutations can be modeled as ordered graphs, where inversions correspond to edges.
 In our case, we add an ordering to multicolored graphs.
Since the ordering of vertices is fixed, the isomorphism testing becomes trivial. On the other hand, the number of possible configurations increases dramatically since each graph is also equipped with a fixed order on its vertices. 

The blow-up argument can be applied to ordered graphs. When we blow up vertices into independent sets, we need the relative order of the independent sets to correspond to the relative order of the original vertices. Hence, we replace each vertex with a consecutive sequence of vertices forming an independent set, see Figure~\ref{fig:ordered-blow-up}. The condition that the independent sets are consecutive can be translated to flag algebras by forbidding two more configurations in the case of 2-edge-colored complete graphs; see Figure~\ref{fig:forbidden-ordered-blow-up}.

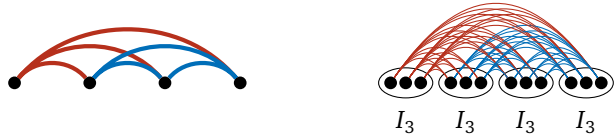
\begin{figure}[H]
    \begin{center}    
\begin{tikzpicture}
\draw
(0,0) node[vertex_u](a0){}
(1,0) node[vertex_u](b0){}
(2,0) node[vertex_u](c0){}
(3,0) node[vertex_u](d0){}
;
\draw[BrickRed,ultra thick]
(a0) to[bend left=50] (b0)
(a0) to[bend left=50] (c0)
(a0) to[bend left=50] (d0)
 ;
\draw[NavyBlue,ultra thick] 
(b0) to[bend left=50] (c0)
(b0) to[bend left=50] (d0)
(c0) to[bend left=50] (d0)
;
\begin{scope}[xshift=5cm,xscale=0.2,yscale=0.5]
\foreach \i in {0,1,2}{
  \draw
 (0+\i,0) node[vertex_u](a\i){}
 (4+\i,0) node[vertex_u](b\i){}
 (8+\i,0) node[vertex_u](c\i){}
 (12+\i,0) node[vertex_u](d\i){}
 ;
}
\foreach \x in {0,1,2,3}{
\draw (4*\x+1,0) ellipse (1.8cm and 0.4 cm);
\draw (4*\x+1,-0.5) node[below]{$I_3$}
;
}

\foreach \i in {0,1,2}{
\foreach \j in {0,1,2}{
 \draw[BrickRed]
 (a\i) to[bend left] (b\j)
 (a\i) to[bend left] (c\j)
 (a\i) to[bend left] (d\j)
 ;
\draw[NavyBlue]  
 (b\i) to[bend left] (c\j)
 (b\i) to[bend left] (d\j)
 (c\i) to[bend left] (d\j)
 ;
 }
 }
\end{scope}
\end{tikzpicture}
    \end{center}
    \caption{An ordered graph on $4$ vertices and its ordered blow-up.}
    \label{fig:ordered-blow-up}
\end{figure}

\newcommand{\OrderedthreeV}[3]{
        \vc{\begin{tikzpicture}[scale=0.75]
        \draw  \foreach \x in {0,1,2} {(\x, 0) node[vertex_u](x\x){}};
        \draw[edge_color#1] (x0)to[bend left=50](x1);
        \draw[edge_color#2] (x0)to[bend left=50](x2);
        \draw[edge_color#3] (x1)to[bend left=50](x2);
        \end{tikzpicture}}
}
\begin{figure}[h!]
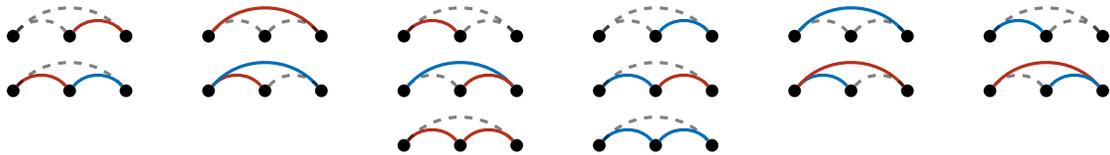

    \centering
    \begin{align*}
        \OrderedthreeV001
        &&
        \OrderedthreeV010        
        && 
        \OrderedthreeV100        
        &&
        \OrderedthreeV002
        &&
        \OrderedthreeV020        
        && 
        \OrderedthreeV200  
        \\
        \OrderedthreeV102
        && 
        \OrderedthreeV120
        && 
        \OrderedthreeV021
        && 
        \OrderedthreeV201
        && 
        \OrderedthreeV210
        && 
        \OrderedthreeV012
        \\
        &&&&
        \OrderedthreeV101
        && 
        \OrderedthreeV202
    \end{align*}
    \caption{The forbidden graphs in blow-ups of $2$-edge colored complete graphs. The first two lines come from orderings of graphs in Figure~\ref{fig:forbidden-blow-up}. The third line comes from each vertex being replaced by an independent set of consecutive vertices. }
    \label{fig:forbidden-ordered-blow-up}
\end{figure}

\section{Lower bounds: integer linear programming formulations and heuristics}

We use integer linear programming (ILP) approaches and tabu search to determine lower bounds on the unordered canonical Ramsey numbers and the ordered Ramsey numbers. While tabu search has the potential to scale to larger graphs, ILPs are more consistent and, in some cases, can be used to obtain sharp upper bounds by infeasibility certificates.  We also formulated the problem as an SAT instance, which we solved using a SAT-solver to obtain feasibility certificates and lower bounds on these numbers. However, this method did not scale well and failed to provide worthwhile bounds. For this reason, we do not include it here. 

\subsection{Integer linear programs for unordered canonical Ramsey numbers}
\label{ssec:ILPunorderedCanonical}

To decide whether the unordered canonical Ramsey number $\CR{G, H}$ is greater than $n$, we formulate an integer linear program that aims to find a coloring of $K_n$ without any orderable subgraph isomorphic to $G$, and any rainbow subgraph isomorphic to $H$. If the program is feasible (and terminates), it returns a Ramsey graph of size $n$, proving that $n+1$ is a lower bound for $\CR{G, H}$. If the program is infeasible, then we know that  $n$ is an upper bound for $\CR{G, H}$. 

We denote by $s$ the size of $G$ and $t$ the size of $H$. We consider the general case where $G$ and $H$ can be any graph, $\CR{s,t}$ being the special case where $G$ and $H$ are both complete.

\paragraph{A first ILP Formulation.}

 We start by introducing a first naive ILP formulation, with a variable for each edge and each possible color. We denote by $\cmax$ the maximum number of colors that can be used in the coloring. To be sure that we consider all possible colorings, $\cmax$ should be set to $\binom{t}{2}-1$, although in practice when a coloring exists, it requires significantly fewer colors. We denote by $\C \coloneqq \{1, \ldots, \cmax\}$ the set of possible colors of the edges.

We introduce the binary variables $x_{ij}^c$ for every $i\neq j \in \nint$ and $c \in \C$, where 
\begin{equation*}
	x_{ij}^c = 
	\begin{cases}
		1 \quad \text{ if edge } \edge ij \text{ has color }c,\\
		0 \quad \text{ otherwise.}
	\end{cases}
\end{equation*}

In practice, we only define these variables for $i<j$, identifying $x_{ij}^c = x_{ji}^c$. For simplicity, we will not concern ourselves with this when describing the constraints here.

To ensure that each edge gets assigned exactly one color, we need to enforce, for every $\edge ij \in \EKn$, that
\begin{equation*}
    \sum\limits_{c \in \C} x_{ij}^c =  1. 
\end{equation*}

To forbid the existence of rainbow subgraphs isomorphic to $H$, we enforce the following constraint. For every copy of $\tilde{H}$ in $K_n$, for all colors $c_{ij}\in \C$, 
\begin{equation*}
    \sum\limits_{\edge ij \in E(\tilde{H})} x_{ij}^{c_{ij}}  \leq \abs{E(H)}-1.
\end{equation*}

The constraints required to enforce the absence of orderable copies of $G$ differ for each graph $G$. We here only explain the ones we introduce for $G=C_3$ and $G=K_4$.

When $G=C_3$, we forbid the presence of orderable triangles: this is equivalent to requiring a proper coloring of the graph. We can enforce this with the following constraint: for all distinct $i,j,k \in \nint$, for every color $c \in \C$,
\begin{equation*}
    x_{ij}^c + x_{ik}^c \leq 1. 
\end{equation*}

For $G=K_4$, the constraint is slightly different. Indeed, the complete graph on $4$ vertices $K_4$ is orderable if and only if it contains one vertex $i$ with all three incident edges of the same color; and another vertex $j$ whose two incident edges different from $\edge ij$ have same color, as depicted in Figure~\ref{fig:ordK4}.

\begin{figure}[ht]
\centering
\begin{subfigure}[b]{.49\linewidth}
\center
\begin{tikzpicture}[scale=0.8]
\draw[thick] (0,0) to (1,0);
\draw[ultra thick, NavyBlue] (1,0) to (1,1);
\draw[ultra thick, BrickRed] (1,1) to (0,1);
\draw[ultra thick, BrickRed] (0,1) to (0,0);
\draw[ultra thick, NavyBlue] (0,0) to (1,1);
\draw[ultra thick, BrickRed] (1,0) to (0,1);

\foreach \i in {0,1} { 
\foreach \j in {0,1} {
\draw (\i,\j) node[vertex_u]{};
}}
\end{tikzpicture}

\end{subfigure}
\hspace{-3cm}
\begin{subfigure}[b]{.49\linewidth}
\center
\begin{tikzpicture}[scale=1]
\draw[ultra thick, BrickRed](0,0) to[bend left=25] (1,0);
\draw[ultra thick, BrickRed](0,0) to[bend left=40] (2,0);
\draw[ultra thick, BrickRed] (0,0) to[bend left=50] (3,0);
\draw[ultra thick, NavyBlue] (1,0) to[bend left=25] (2,0);
\draw[ultra thick, NavyBlue] (1,0) to[bend left=40] (3,0);
\draw[thick] (2,0) to[bend left=25] (3,0);

\foreach \i in {0,...,3} {
\draw (\i,0) node[vertex_u]{};
}
\end{tikzpicture}
\end{subfigure}
\caption{An orderable coloring of $K_4$.}
\label{fig:ordK4}
\end{figure}
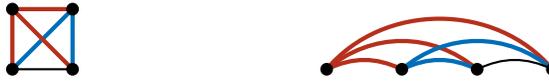

We can thus forbid this with the following constraint: for all pairwise-distinct vertices $i,j,k,\ell$,

all colors $c,d\in \C$,
\begin{equation*}
 x_{ij}^c + x_{ik}^c + x_{i\ell}^c + x_{jk}^d + x_{j\ell}^d \leq 4.  \label{ct:ordK4} \\
\end{equation*}

For $G=C_3$, we then obtain the following set of constraints
\begin{subequations}
	\label{eq:formulationUCR1}
	\begin{alignat}{3}
 &&  \sum\limits_{c \in \C} x_{ij}^c &=  1 && \forall \edge ij \in \EKn \nonumber\\       	
 && \sum\limits_{\edge ij \in  E(\tilde{H})} x_{ij}^{c_{ij}} & \leq\abs{E(H)}-1 \quad && \forall \ \tilde{H} \cong H, \ \forall c_{ij}\in \C \label{ct:rainH}\\
  && x_{ij}^c + x_{ik}^c &\leq 1 && \forall \text{distinct }i,j,k \in \nint, \forall  c \in \C \label{ct:ordTriangle}\\
	&&x_{ij}^c  &\in \{0,1\} && \forall i\neq j \in \nint, \forall c \in C. \nonumber
	\end{alignat}
\end{subequations}
The formulation is similar for any other forbidden orderable graph $G$, with the constraint~\eqref{ct:ordTriangle} being replaced by the appropriate constraint; for instance, constraint~\eqref{ct:ordK4} for $G=K_4$.

We can also choose to minimize (resp. maximize) the number of colors used by introducing variables $(z^c)_{c \in \C}$ that take value $1$ if and only if the color $c$ is used at least once in the coloring. We then minimize (resp. maximize) $\sum_{c \in \C} z^c$ and add the constraint $z^c\geq x_{ij}^c$ for every $\edge ij \in E(G)$ and~$c \in \C$ (resp. $z^c\leq \sum_{ij} x_{ij}^c$ for every $c \in \C$).

This program works well on very small instances. However, considering the exponential number of constraints, it is no surprise that this program quickly becomes virtually impossible to solve. In particular, there are $\cmax^{ t}\cdot \binom{n}{t}$ constraints of type~\eqref{ct:rainH}, forbidding the presence of a rainbow copy of $K_t$. This number becomes quickly so large that it is impossible even to generate all of the constraints.

\paragraph{A Colorblind ILP formulation.}

To tackle this scalability issue, we introduce another integer linear formulation for this problem, where we do not have to generate all possible combinations of colors: a colorblind formulation.

We introduce binary variables $y_{e,f}$ for all distinct edges $e,f \in \EKn$, such that \begin{equation*}
	y_{e,f} = 	
	\begin{cases}
		1 \quad \text{ if edges } e \text{ and } f \text{ have the same color},\\
		0 \quad \text{ otherwise.}
	\end{cases}
\end{equation*}

We identify $y_{e,f}=y_{f,e}$ for every pair of edges $(e,f)$, and, as before, always write edges $\edge ij$ without regard to the order of vertices. Note that we have the following relation between these variables and the $x$-variables previously introduced: for all distinct edges $e,f$ \begin{equation*} 
y_{e,f} = \sum\limits_{c \in \C} x_e^cx_f^c.
\end{equation*}
To accurately model a coloring, these variables have to respect transitivity: for any three edges $e,f,g$, if edges $e$ and $g$ and edges $f$ and $g$ have the same color, then edges $e$ and $f$ have the same color. We enforce these relations using the \emph{transitivity inequalities}, for all pairwise-distinct edges~$e,f,g$ 
\begin{equation*}
\label{eq:transitivity}
  y_{e,f} - y_{e,g} -y_{f,g} \geq  - 1.
\end{equation*}

With these variables, the absence of a rainbow copy of $H$ can be enforced by requiring at least two edges of each copy $\tilde{H}$ of $H$ in $K_n$ to be of the same color:
\begin{equation*}
	\sum\limits_{\substack{e,f \in E(\tilde{H})}} y_{e,f}  \geq 1. 
\end{equation*}

We now detail how we forbid orderable copies of graph $G$ depending on $G$ in the cases where $G$ is a complete graph or a cycle.

A colored cycle is orderable if and only if it has one vertex with two incident edges of the same color, as shown in Figure~\ref{fig:C6_orderable_colorings}.

Hence, forbidding canonical colorings of any cycle is equivalent to requiring proper edge coloring (if the graph has at least as many vertices as the cycle).

This can be expressed by fixing \begin{equation*}
    y_{ij,ik} = 0
\end{equation*}
for all vertices $i,j,k \in \nint$.

\begin{figure}[h]
\centering
\begin{subfigure}[b]{.49\linewidth}
\center
\begin{tikzpicture}[scale=0.6]
\foreach \i in {0,...,5}{
\draw (60*\i:1) coordinate(\i);
}

\draw[thick](0)--(1)--(2)--(3)--(4)--(5)-- cycle;
\draw[line width=2pt, BrickRed, line cap=round] (2) -- (3) -- (4);

\foreach \i in {0,...,5}{
\draw (\i) node[vertex_u]{};
}

\end{tikzpicture}

\end{subfigure}
\hspace{-3cm}
\begin{subfigure}[b]{.49\linewidth}
\center
\begin{tikzpicture}[scale=0.8]
\draw[ultra thick, BrickRed](0,0) to[bend left=30] (1,0);
\draw[thick] (1,0) to[bend left=30] (2,0);
\draw[thick] (2,0) to[bend left=30] (3,0);
\draw[thick] (3,0) to[bend left=30] (4,0);
\draw[thick] (4,0) to[bend left=30] (5,0);
\draw[ultra thick, BrickRed] (0,0) to[bend left=30] (5,0);

\foreach \i in {0,...,5} {
\draw (\i,0) node[vertex_u]{};
}
\end{tikzpicture}
\end{subfigure}
\caption{An orderable coloring of $C_6$. The black edges can be any color.}
\label{fig:C6_orderable_colorings}
\end{figure}
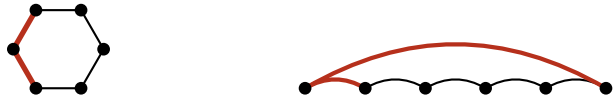

\paragraph{}
As mentioned in the previous section, the complete graph $G=K_4$ is orderable if and only if it contains one vertex $i$ with all three incident edges of the same color and another vertex $j$ whose two incident edges are different from $\edge ij$ have the same color (see Figure~\ref{fig:ordK4}). We can express this constraint as 
\begin{equation*}
	y_{ij, ik} + y_{ij, i\ell} + y_{ik, i\ell} + y_{jk,j\ell} \leq 3\qquad \text{for all pairwise-distinct } i,j,k,\ell \in \nint.
\end{equation*}

This constraint enforces that, in the clique formed by vertices $\{i,j,k,\ell\}$, vertex $i$ can have three incident edges of the same color if and only if $j$ does not have its other two incident edges of the same color. Enforcing this for all pairwise-distinct $i,j,k,\ell$ ensures that there is no orderable clique in the graph.

Likewise, a coloring of $K_p$ is orderable if and only if it has one vertex with $n$ incident edges of the same color, another vertex with $(n-1)$ incident edges of the same color, another one with $(n-2)$ incident edges of the same color, and so on. More orderability constraints for bipartite graphs are detailed in Appendix~\ref{app:constraintOrderabilityBipartite}.

For $G=C_3$ we obtain the following set of constraints
\begin{align}
& & y_{e,f} - y_{e,g} -y_{f,g}&\geq  - 1  \quad && \forall e,f,g \in \EKn, \nonumber\\ 
&&     \sum\limits_{j,k\in \nint\setminus\{i\}} y_{\edge ij,\edge ik}  &= 0 
 \quad && \forall i \in \nint,  \label{ct:clb_ordTriangle}\\
&&  \sum\limits_{\substack{e,f \in E(\tilde{H})}} y_{e,f}  &\geq 1&& \forall \ \tilde{H} \cong H  \subseteq K_n, \nonumber\\
&& \qquad \quad y_{e,f} \in \ &\{0,1\} && \forall e,f \in \EKn. \nonumber
\end{align}

To forbid other graphs $G$ to be orderable, we replace constraint~\eqref{ct:clb_ordTriangle} with one of the other constraints described above or in the Appendix~\ref{app:constraintOrderabilityBipartite}.

Compared to Formulation~\eqref{eq:formulationUCR1}, the number of variables grows from order $n^2$ to $n^4$ and the number of constraints remains exponential. However, the colorblind formulation has the advantage of avoiding the bottleneck of the previous formulation, which lies in computing all possible rainbow colorings of a clique of size $t$. Though this formulation still does not allow us to obtain bounds for large values of $n$, especially as $t$ increases, it provides more results, and in less time, than the non-colorblind one.

\subsection{Integer linear formulation for ordered Ramsey numbers}
\label{ssec:ILPorderedRamsey}

Here we introduce an ILP formulation that, for given $n$ and $G$, finds a $2$-coloring of $K_n$ that does not contain a monochromatic copy of $G$ when it exists, implying that $ \ordRam{\G} \geq n+1$. If this program is not feasible, we obtain that $\ordRam{G} \leq n$.

To this end, we model a red-blue coloring of the edges of $K_n$ by introducing a vector of variables~$x \in \set{0,1}^{\EKn}$, where, for every edge $e \in \EKn$,

\begin{equation*}
	x_e = 	
	\begin{cases}
		1 \quad \text{ if } e \text{ has color red,} \\
		0 \quad \text{ if } e \text{ has color blue.}
	\end{cases}
\end{equation*}

We then formulate the following set of linear constraints
\begin{subequations}
	\label{eq:ORN_ILP}
	\begin{alignat}{3}
 && \sum\limits_{e \in  E(\tilde{G})} x_{e} & \geq 1 \quad && \forall \ \tilde{G} \cong G, \label{ct:notallblue}\\
  && \sum\limits_{e \in  E(\tilde{G})} x_{e} & \leq \abs{E(G)}-1 \quad && \forall \ \tilde{G} \cong G, \label{ct:notallred} \\
	&& x_{e} \in &\set{0,1} && \forall e \in  \EKn,
	\end{alignat}
\end{subequations}
in which constraint~\eqref{ct:notallblue} (resp.~\eqref{ct:notallred}) forbid edges of a copy of $G$ in $K_n$ to be all blue (resp. red).

The number of constraints of this formulation remains exponential in $\abs{G}$, 

preventing us from obtaining bounds for larger graphs. However, it is efficient in computing bounds for small graphs, as shown in Section~\ref{ssec:boundsORN}.

\subsection{Tabu search}
\label{ssec:tabu}

We implemented a tabu heuristic similar to what was done by \Lidicky{}, McKinley, and Pfender~\cite{Lidicky2024} to construct feasible graphs for Ramsey numbers.
The goal of this method is to construct a coloring of $K_n$ heuristically without the forbidden patterns. 

The idea is the following: we implement a \emph{scoring function}, which returns how far we are from a Ramsey graph. In practice, the scoring function simply counts the number of copies of forbidden graphs appearing in a given graph. Then, starting from a random coloring, we iteratively recolor a single edge. Here, we pick the edge that leads to the largest improvement in the scoring function. To 

avoid getting stuck in a local minimum,
the \emph{tabu list} comes into play, preventing the algorithm from visiting graphs more than once. This means that in some steps, the score of the graph will first get worse before it gets better. Here, we only store the hashes of the graphs in the tabu list. We stop the algorithm when the score of the graph reaches zero, i.e., it does not contain any forbidden graphs. 

We can speed up this approach by computing an upper bound on the number of colors. This can be done by maximizing the number of colors in a small Ramsey graph using an ILP, as explained in Section~\ref{ssec:ILPunorderedCanonical}.

This approach was needed to obtain the sharp lower bound $\CR{3,5}\geq 13$.

\subsection{Lower bounds by complete enumeration}
\label{ssec:enumeration}

One can also try to generate all Ramsey graphs of size $n$ comprehensively. We start by generating all Ramsey graphs for $n=1$, that is, only the graph on one vertex. We then proceed iteratively. For each $n \geq 1$, let us denote by $\set{G_1, \ldots, G_m}$ the set of feasibly colored graphs of size $n$ up to isomorphism. From each $G_i$, we construct a graph of size $n+1$ by attaching a new vertex $n+1$. We then consider all the possible colorings of the edges incident to the vertex $n+1$ and keep only the ones that are feasible. The union of these graphs obtained from each $G_i$ contains the set of all feasible colorings of $K_{n+1}$ up to isomorphism. Keeping only a set of representatives of these graphs, we repeat this iteratively until we reach a size where no Ramsey graphs can be found. The limitation of this method clearly lies in the fact that the number of Ramsey graphs may first become extremely large before becoming smaller again, and at some point, it may become practically impossible to compute the next iteration of the algorithm. 

This is, for instance, how we compute the sharp lower bound $\CR{6,3}\geq 26$, as well as many of the smaller lower bounds seen in Tables~\ref{tab:boundsOrdCanRamseyNumbers} and~\ref{tab:boundsOrdRamseyNumbers}.

For $\CR{6,3}$, it was not possible to enumerate all Ramsey graphs up to order 25 directly. First, we used flag algebras calculation on 6 vertices to show that $\CR{6,3} \leq 26$. Since the calculation seemed sharp, we inspected the solution and noticed that monochromatic triangles were missing as well as six other colored graphs of order 6.

After forbidding these 7 graphs in addition to orderable $K_6$ and rainbow $C_3$, 
we were able to enumerate the graphs up to order 25 and found some graphs on 25 vertices.   
See Table~\ref{tab:nbFeasibleGraphsCR63}, where give the number of these graphs for for each size of graph $n$. The enumeration implies $\CR{6,3} \geq 26$.
This demonstrates how even a numerical solution using flag algebras can help with direct enumeration.

\begin{table}[H]
\centering
\begin{tabular}{@{}lrrrrrrrrr@{}}
\toprule
$n$      & 1     & 2     & 3     & 4    & 5     & 6     & 7     & 8     & 9   \\ 
\#graphs & 1     & 1     & 1     & 4    & 9     & 26    & 59    & 164   & 376 \\ \midrule \midrule
$n$      & 10    & 11    & 12    & 13   & 14    & 15    & 16    & 17    &  18   \\ 
\#graphs & 901   & 1869  & 3727  & 6316 & 10482 & 15754 & 21319 & 25275 &  25269  \\ \midrule
\midrule
$n$       & 19    & 20    & 21   & 22    & 23    & 24    & 25    &  26&   \\ 
\#graphs  & 21901 & 15483 & 9433 & 4411  & 1795  & 476   & 127   &  0  & \\ \bottomrule
\end{tabular}
\caption{Number of graphs (up to isomorphism) used in enumeration for showing $\CR{6,3} = 26$.}
\label{tab:nbFeasibleGraphsCR63}
\end{table}

\section{Results}

We now present the bounds obtained through the methods introduced above for the three Ramsey variants considered. Throughout this section, we mark bounds in blue when they are sharp and otherwise provide the interval [lower bound, upper bound].

\subsection{Canonical Ramsey numbers}

In Table~\ref{tab:BoundsUCRNComplete}, we present the bounds we obtain for the (ordered) canonical Ramsey numbers $\ER{G}$ for all ordered graphs up to $4$ vertices without isolated vertices, up to reflection. Indeed, the ordered Ramsey number of the reflection of a graph $\G$, that is, the graph with the reverse ordering, is equal to the one of $\G$.

All sharp upper bounds are obtained through complete enumeration or infeasibility of the corresponding ILP formulation,

and other upper bounds are obtained using flag algebras. All computations were done on $5$ vertices except the upper bound for $K_4$, where the SDP was too large to be solved.

\begin{table}[H]
\centering
	\begin{tabular}{l|ccccccccc}\toprule
		 
		  $K_2$ & $\ramseybound{\Orderedtwo1}{2}{2}$ & \\\midrule
        
		$P_3$ & $\ramseybound{\Orderedthree011}{3}{3}$ &
		$\ramseybound{\Orderedthree101}{3}{3}$ & \\\midrule
		
        $C_3$ & $\ramseybound{\Orderedthree111}{4}{4}$ & \\\midrule
		
		$2K_2$ & $\ramseybound{\Orderedfour001100}{4}{4}$ & 
		$\ramseybound{\Orderedfour010010}{4}{4}$ & 
        $\ramseybound{\Orderedfour100001}{4}{4}$ & \\\midrule
		
		$P_4$ & $\ramseybound{\Orderedfour001101}{7}{7}$ & 
		$\ramseybound{\Orderedfour001110}{5}{5}$ & 
		$\ramseybound{\Orderedfour010011}{6}{6}$ & 
		$\ramseybound{\Orderedfour010110}{5}{5}$ & 
		$\ramseybound{\Orderedfour011010}{5}{5}$ & 
		$\ramseybound{\Orderedfour100011}{6}{6}$ & \\
		& $\ramseybound{\Orderedfour100101}{10}{10}$ &
        $\ramseybound{\Orderedfour101001}{6}{6}$ & \\\midrule
        
		$K_{1,3}$ & $\ramseybound{\Orderedfour001011}{6}{6}$ & 
		$\ramseybound{\Orderedfour010101}{5}{5}$ & \\\midrule
        
		Paw & $\ramseybound{\Orderedfour001111}{9}{9}$ &
		$\ramseybound{\Orderedfour010111}{12}{43}$ &
		$\ramseybound{\Orderedfour011011}{9}{9}$ &
		$\ramseybound{\Orderedfour011101}{11}{26}$ &
		$\ramseybound{\Orderedfour100111}{11}{17}$ &
		$\ramseybound{\Orderedfour101011}{10}{10}$ \\ \midrule
        
		$C_4$ & $\ramseybound{\Orderedfour011110}{9}{14}$ &
		$\ramseybound{\Orderedfour101101}{9}{37}$ &
		$\ramseybound{\Orderedfour110011}{11}{24}$\\\midrule
        
		$K_4^-$ & $\ramseybound{\Orderedfour011111}{15}{\infty}$ &
		$\ramseybound{\Orderedfour101111}{15}{5224}$   &
		$\ramseybound{\Orderedfour110111}{15}{\infty}$ &
		$\ramseybound{\Orderedfour111011}{13}{288}$   \\\midrule
		
        $K_4$ & $\ramseybound{\Orderedfour111111}{16}{\infty}$\\
	\bottomrule
	\end{tabular}
    \caption{Bounds for ordered canonical Ramsey numbers $\ER{G}$. }
    \label{tab:boundsOrdCanRamseyNumbers}
\end{table}

\subsection{Unordered canonical Ramsey numbers}

\paragraph*{Complete graphs.}

In Table~\ref{tab:BoundsUCRNComplete}, we present our results for the unordered canonical Ramsey numbers $\CR{s,t}$ for $3\leq s\leq 6$ and $3\leq t \leq 5$. We obtain exact results for six of these numbers and we improve the known upper bound on $\CR{4,4}$. The other numbers are still, for the moment, out of reach for our algorithms.

\begin{table}[H]
\centering

\begin{tabular}{l|ccc}
    \toprule
     ${}_{s}$ ${}^{t}$ & 3  & 4  &  5  \\ \midrule
      3   &    \sharpboundV{3}&    \sharpboundV{7}&  \sharpboundV{13}\\
      4   &    \sharpboundV{6}&  $[26,45]$  &   -  \\
      5   &  \sharpboundV{11} &  -  &  -   \\          
      6   & \sharpboundV{26}   &   - &  -   \\  \bottomrule        
\end{tabular}
\caption{Lower and upper bounds on unordered canonical Ramsey numbers $\CR{s,t}$ of complete graphs.
}
\label{tab:BoundsUCRNComplete}
\end{table}

These results were obtained in the following way: solving the ILP provided the exact values of $\CR{3,3}$, $\CR{3,4}$ and $\CR{4,3}$; as well as the lower bound on $\CR{5,3}$.
We found a construction for $\CR{3,5}$ on $12$ vertices with the tabu algorithm, providing us with the lower bound of $13$.
The construction of a graph on $25$ vertices for $\CR{6,3}$ was obtained through complete enumeration. The lower bound for $\CR{4,4}$ is the specific case of the bound proven by Richer~\cite{Richer2000} for all $s$ and $t$, as shown in Equation~\eqref{eq:boundsUCRN}.

The upper bounds for $\CR{3,5}$, $\CR{4,4}$, $\CR{5,3}$ and $\CR{6,3}$ were all obtained with the colorblind flag algebras method, on $7$, $6$, $5$ and $6$ vertices respectively.
The computation of the upper bound for $\CR{4,4}$ required 505,163 unlabeled flags in the program, requiring just below $2$TB of memory 

using \texttt{csdp}.

\paragraph*{Bipartite graphs.}

We now present our bounds on the unordered canonical Ramsey number $\CR{G,H}$ for bipartite graphs. A colored graph $G$ on $n$ vertices is orderable if and only if the restriction of $G$ to its vertices of degree at least $2$ is orderable. Hence, for any $G$, an edge-coloring of $K_n$, with $n>\abs{V(G)}$, does not contain an orderable $G$ if and only if it does not contain an orderable $G$ (recursively) removed of all its vertices of degree $1$. In our computations, we thus only consider the unordered canonical Ramsey numbers $\CR{G, H}$ for bipartite $G$ without vertices of degree $1$. Note that since we can apply this recursively, this implies, in particular, that any coloring of a tree is orderable. The graphs we examine are the even cycles $C_4$ and $C_6$, the complete bipartite graphs $\Kb23$, $\Kb23$, $\Kb24$ and $\Kb33$, the ladder graph on $6$ vertices $L_3$ and the complete bipartite graph with one edge removed $\Kbmin33$. They are represented in Figure~\ref{fig:bipartiteGraphs}.
All the lower bounds presented here are obtained by the colorblind ILP formulation of Section~\ref{ssec:ILPunorderedCanonical}; and the upper bounds are obtained via flag algebra computations. They are summarized in Table~\ref{tab:boundsUnorderedCanonicalBipartite}.

\begin{figure}[H]
\centering
\begin{subfigure}[b]{.13\linewidth}
\center
\begin{tikzpicture}[scale=0.5]
\foreach \x/\y in {0/0, 0/1, 1/1, 1/0  }{\draw (\x,\y) node[vertex_su]{};  }   
\draw[thick] (0,0) to (1,0);
\draw[thick] (1,0) to (1,1);
\draw[thick] (1,1) to (0,1);
\draw[thick] (0,1) to (0,0);
\end{tikzpicture}
\caption*{$C_4$}\label{fig:C4}
\end{subfigure}
\begin{subfigure}[b]{.13\linewidth}
\center
\begin{tikzpicture}[scale=0.5]
\foreach \x/\y in {0/0, 0/1, 1/-0.5, 1/0.5, 1/1.5  }{\draw (\x,\y) node[vertex_su]{};  }   
\draw[thick] (0,1) to (1,-0.5);
\draw[thick] (0,1) to (1,0.5);
\draw[thick] (0,1) to (1,1.5);
\draw[thick] (0,0) to (1,-0.5);
\draw[thick] (0,0) to (1,0.5);
\draw[thick] (0,0) to (1,1.5);
\end{tikzpicture}\caption*{$\Kb{2}{3}$}\label{fig:K23}
\end{subfigure}
\begin{subfigure}[b]{.13\linewidth}
\center
\begin{tikzpicture}[scale=0.5]
\foreach \x/\y in {0/0, 0.4/0.6, 0.4/-0.6, 1/0.6, 1/-0.6, 1.4/0  }{\draw (\x,\y) node[vertex_su]{};  }   
\draw[thick] (0,0) to (0.4,0.6);
\draw[thick] (0,0) to (0.4,-0.6);
\draw[thick] (0.4,0.6) to (1,0.6);
\draw[thick] (0.4,-0.6) to (1,-0.6);
\draw[thick] (1,0.6) to (1.4,0);
\draw[thick] (1,-0.6) to (1.4,0);
\end{tikzpicture}\caption*{$C_6$}\label{fig:C6}
\end{subfigure}
\begin{subfigure}[b]{.13\linewidth}
\center
\begin{tikzpicture}[scale=0.5]
\foreach \x/\y in {0/0, 1/0, 1/1, 0/1, 2/0, 2/1  }{\draw (\x,\y) node[vertex_su]{};  }   
\draw[thick] (0,0) to (1,0);
\draw[thick] (1,0) to (1,1);
\draw[thick] (1,1) to (0,1);
\draw[thick] (0,1) to (0,0);
\draw[thick] (1,0) to (2,0);
\draw[thick] (2,0) to (2,1);
\draw[thick] (2,1) to (1,1);
\end{tikzpicture}\caption*{$L_3$}\label{fig:L3}
\end{subfigure}
\begin{subfigure}[b]{.13\linewidth}
\center
\begin{tikzpicture}[scale=0.5]
\foreach \x/\y in {0/0, 0/1, 0/2, 1/0, 1/1, 1/2  }{\draw (\x,\y) node[vertex_su]{};  }   
\draw[thick] (0,2) to (1,2);
\draw[thick] (0,2) to (1,1);
\draw[thick] (0,2) to (1,0);
\draw[thick] (0,1) to (1,2);
\draw[thick] (0,1) to (1,1);
\draw[thick] (0,1) to (1,0);
\draw[thick] (0,0) to (1,2);
\draw[thick] (0,0) to (1,1);

\end{tikzpicture}\caption*{$\Kbmin{3}{3}$}\label{fig:K33min}
\end{subfigure}
\begin{subfigure}[b]{.13\linewidth}
\center
\begin{tikzpicture}[scale=0.5]
\foreach \x/\y in {0/2, 0/1,  1/2.5, 1/1.85, 1/0.5,1/1.15  }{\draw (\x,\y) node[vertex_su]{};  }   

\draw[thick] (0,2) to (1,2.5);
\draw[thick] (0,2) to (1,1.85);
\draw[thick] (0,2) to (1,1.15);
\draw[thick] (0,2) to (1,0.5);
\draw[thick] (0,1) to (1,2.5);
\draw[thick] (0,1) to (1,1.85);
\draw[thick] (0,1) to (1,1.15);
\draw[thick] (0,1) to (1,0.5);
\end{tikzpicture}\caption*{$\Kb{2}{4}$}\label{fig:K24}
\end{subfigure}
\begin{subfigure}[b]{.13\linewidth}
\center
\begin{tikzpicture}[scale=0.5]
\foreach \x/\y in {0/0, 0/1, 0/2, 1/0, 1/1, 1/2  }{\draw (\x,\y) node[vertex_su]{};  }   
\draw[thick] (0,2) to (1,2);
\draw[thick] (0,2) to (1,1);
\draw[thick] (0,2) to (1,0);
\draw[thick] (0,1) to (1,2);
\draw[thick] (0,1) to (1,1);
\draw[thick] (0,1) to (1,0);
\draw[thick] (0,0) to (1,2);
\draw[thick] (0,0) to (1,1);
\draw[thick] (0,0) to (1,0);
\end{tikzpicture}\caption*{$\Kb{3}{3}$}\label{fig:K33}
\end{subfigure}
\caption{The bipartite graphs considered.}
\label{fig:bipartiteGraphs}
\end{figure}
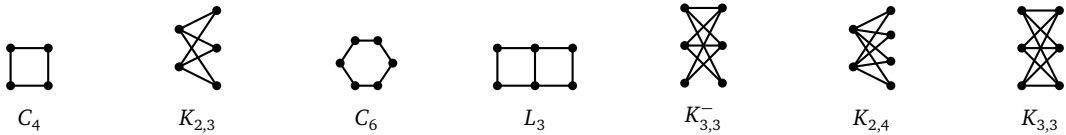

\begin{table}[H]
\centering
	\begin{tabular}{l|ccccccc}\toprule
		 ${}_{G}$ ${}^{H}$& $C_4$& $\Kb{2}{3}$ & $C_6$ & $L_3$ & $\Kbmin{3}{3}$ &$\Kb{2}{4}$ &  $\Kb{3}{3}$\\\midrule
		$C_4$ & \sharpboundV{5 }& \sharpboundV{7 }& \sharpboundV{9 }& \sharpboundV{9 }& \sharpboundV{10} & \sharpboundV{10} & $[11,13]$\\
		$\Kb{2}{3}$ & \sharpboundV{7 }& \sharpboundV{9 }& \sharpbound{9 }& \sharpbound{9 }& $[\boundV{11},12]$ & $[11,13]$ &  $[11,12]$\\
		$C_6$ & \sharpboundV{6 } & \sharpboundV{7 }& \sharpboundV{9 }& \sharpboundV{9 }& \sharpboundV{10} & \sharpboundV{10} & $[11,13]$\\
		$L_3$ & \sharpboundV{6 }& \sharpboundV{8 }& \sharpboundV{10} & \sharpbound{10} & \sharpbound{11} & $[\boundV{11},12]$ & $[\boundV{12},15]$\\
		$\Kbmin{3}{3}$ & \sharpboundV{7 }& $\sharpboundV{10}$ & $\sharpboundV{10}$ & $[\boundV{10},11]$ & $[10,\infty)$ & $[\boundV{11},\infty)$ &       $[\boundV{12},\infty)$\\
		$\Kb{2}{4}$ & \sharpboundV{10} & $[10,\infty)$ & $[10,\infty)$ & $[\boundV{12},\infty)$ & $[\boundV{12},\infty)$ & $[12,\infty)$&          $[\boundV{13},\infty)$\\
		 $\Kb{3}{3}$ & \sharpbound{10} & $[\boundV{12},\infty)$ & $[11,\infty)$ & $[11,\infty)$ & $[\boundV{13},\infty)$ & $[\boundV{12},\infty)$&          $[\boundV{14},\infty)$ \\
	\bottomrule
	\end{tabular}
    \caption{Bounds on $\CR{G,H}$ for bipartite.}
    \label{tab:boundsUnorderedCanonicalBipartite}
\end{table}

\subsection{Ordered Ramsey numbers}
\label{ssec:boundsORN}

In Table~\ref{tab:boundsOrdRamseyNumbers}, we display the bounds obtained by the flag algebras on the ordered Ramsey numbers $\ordRam{\G}$ of all graphs of size up to $4$, up to reflection. The lower bounds shown here are the known exact Ramsey numbers, proven in~\cite{Balko2015} and~\cite{Overman2018}.

\begin{table}[H]
\centering
	\begin{tabular}{l|ccccccccc}\toprule
		$K_2$ & 
        $\ramseybound{\Orderedtwo1}{2}{2}$\\\midrule
		$P_3$ & 
        $\ramseybound{\Orderedthree011}{4}{4}$ & 
        $\ramseybound{\Orderedthree101}{5}{5}$\\\midrule
		$C_3$ & 
        $\ramseybound{\Orderedthree111}{6}{6}$\\\midrule
		$2K_2$ & 
        $\ramseybound{\Orderedfour001100}{6}{6}$ & 
        $\ramseybound{\Orderedfour010010}{5}{5}$ & 
        $\ramseybound{\Orderedfour100001}{6}{6}$\\\midrule
		$P_4$ & 
        $\ramseybound{\Orderedfour001101}{9}{9}$ & 
        $\ramseybound{\Orderedfour001110}{7}{7}$ & 
        $\ramseybound{\Orderedfour010011}{9}{9}$ & 
        $\ramseybound{\Orderedfour010110}{7}{7}$ & 
        $\ramseybound{\Orderedfour011010}{7}{7}$ & 
        $\ramseybound{\Orderedfour100011}{9}{9}$ \\ & 
        $\ramseybound{\Orderedfour100101}{10}{10}$ & 
        $\ramseybound{\Orderedfour101001}{10}{10}$\\\midrule
		$K_{1,3}$ & 
        $\ramseybound{\Orderedfour001011}{6}{6}$ & 
        $\ramseybound{\Orderedfour010101}{9}{9}$\\\midrule
		Paw & 
        $\ramseybound{\Orderedfour001111}{10}{10}$ & 
        $\ramseybound{\Orderedfour010111}{10}{10}$ & 
        $\ramseybound{\Orderedfour011011}{10}{10}$ & 
        $\ramseybound{\Orderedfour011101}{11}{11}$ & 
        $\ramseybound{\Orderedfour100111}{11}{11}$ & 
        $\ramseybound{\Orderedfour101011}{10}{10}$\\\midrule
		$C_4$ & 
        $\ramseybound{\Orderedfour011110}{10}{10}$ & 
        $\ramseybound{\Orderedfour101101}{14}{14}$ & 
        $\ramseybound{\Orderedfour110011}{11}{11}$\\\midrule 
		$K_4^-$ & 
        $\ramseybound{\Orderedfour011111}{12}{12}$ & 
        $\ramseybound{\Orderedfour101111}{14}{17}$ & 
        $\ramseybound{\Orderedfour110111}{15}{16}$ & 
        $\ramseybound{\Orderedfour111011}{13}{14}$\\\midrule
		$K_4$ & 
        $\ramseybound{\Orderedfour111111}{18}{18}$\\ 
	\bottomrule
	\end{tabular}
    \caption{Bounds for ordered Ramsey numbers $\ordRam{G}$.}
    \label{tab:boundsOrdRamseyNumbers}
\end{table}

We show in Table~\ref{fig:ordRam5} our bounds on the ordered Ramsey numbers $\ordRam{\G}$ on some selected ordered graphs of size $5$. The second and third columns show the bounds obtained with the ILP formulation introduced in Section~\ref{ssec:ILPorderedRamsey}. The last two columns display the upper bounds obtained using the ordered variant of the flag algebra method described in Section~\ref{ssec:orderedFA} on $5$ and $6$ vertices, respectively.

\begin{table}[H]
\centering
\begin{tabular}{cccccc}
\toprule
 & & & & UB FA & UB FA   \\ 
ID &$\G$ & LB ILP & UB ILP & $n = 5$ & $n = 6$  \\ \midrule
1 &
\begin{tikzpicture}[scale=0.5]
\foreach \x in {0,1,2,3,4}{\draw (\x,0) node[vertex_su]{};}
\draw[thick] (0,0) to[bend left=50] (1,0);
\draw[thick] (0,0) to[bend left=50] (2,0);
\draw[thick] (1,0) to[bend right=50] (3,0);
\draw[thick] (2,0) to[bend left=50] (4,0);
\end{tikzpicture}
 & \sharpbound{14} & \sharpbound{14} & 27 & 19 \\ \hline
 2 & 
\begin{tikzpicture}[scale=0.5]
\foreach \x in {0,1,2,3,4}{\draw (\x,0) node[vertex_su]{};}
\draw[thick] (0,0) to[bend left=50] (1,0);
\draw[thick] (0,0) to[bend left=50] (2,0);
\draw[thick] (1,0) to[bend right=50] (3,0);
\draw[thick] (3,0) to[bend left=50] (4,0);
\end{tikzpicture}
 & \sharpbound{16} & \sharpbound{16} & 45 & 25
\\ \hline
3 &
\begin{tikzpicture}[scale=0.5]
\foreach \x in {0,1,2,3,4}{\draw (\x,0) node[vertex_su]{};}
\draw[thick] (0,0) to[bend left=50] (1,0);
\draw[thick] (0,0) to[bend left=50] (2,0);
\draw[thick] (1,0) to[bend right=50] (4,0);
\draw[thick] (2,0) to[bend left=50] (3,0);
\end{tikzpicture}
 & \sharpbound{15} & \sharpbound{15} & 32 & 20
\\ \hline
4 &
\begin{tikzpicture}[scale=0.5]
\foreach \x in {0,1,2,3,4}{\draw (\x,0) node[vertex_su]{};}
\draw[thick] (0,0) to[bend left=50] (2,0);
\draw[thick] (0,0) to[bend left=50] (4,0);
\draw[thick] (1,0) to[bend right=50] (3,0);
\draw[thick] (1,0) to[bend right=50] (4,0);
\draw[thick] (2,0) to[bend left=50] (3,0);
\end{tikzpicture}
 & \sharpbound{14} & \sharpbound{14} & 31 & 29
\\ \hline
5 &
\begin{tikzpicture}[scale=0.5]
\foreach \x in {0,1,2,3,4}{\draw (\x,0) node[vertex_su]{};}
\draw[thick] (0,0) to[bend left=50] (1,0);
\draw[thick] (0,0) to[bend left=50] (2,0);
\draw[thick] (1,0) to[bend right=50] (3,0);
\draw[thick] (2,0) to[bend left=50] (4,0);
\draw[thick] (3,0) to[bend right=50] (4,0);
\end{tikzpicture}
 & 19 & - & 50 & 30
\\ \hline
6 &
\begin{tikzpicture}[scale=0.5]
\foreach \x in {0,1,2,3,4}{\draw (\x,0) node[vertex_su]{};}
\draw[thick] (0,0) to[bend left=50] (1,0);
\draw[thick] (0,0) to[bend left=50] (2,0);
\draw[thick] (1,0) to[bend right=50] (2,0);
\draw[thick] (1,0) to[bend right=50] (4,0);
\draw[thick] (2,0) to[bend left=50] (3,0);
\draw[thick] (3,0) to[bend left=50] (4,0);
\end{tikzpicture}
& 24 & - & 
- 
& 33
\\ \hline
7 &
\begin{tikzpicture}[scale=0.5]
\foreach \x in {0,1,2,3,4}{\draw (\x,0) node[vertex_su]{};}
\draw[thick] (0,0) to[bend left=50] (1,0);
\draw[thick] (1,0) to[bend right=50] (2,0);
\draw[thick] (1,0) to[bend right=50] (3,0);
\draw[thick] (1,0) to[bend right=50] (4,0);
\end{tikzpicture}
 & \sharpbound{13} & \sharpbound{13} & 21 & 16
\\ \hline
8 &
\begin{tikzpicture}[scale=0.5]
\foreach \x in {0,1,2,3,4}{\draw (\x,0) node[vertex_su]{};}
\draw[thick] (2,0) to[bend left=50] (0,0);
\draw[thick] (2,0) to[bend left=50] (1,0);
\draw[thick] (2,0) to[bend left=50] (3,0);
\draw[thick] (2,0) to[bend left=50] (4,0);
\end{tikzpicture}
 & \sharpbound{15} & \sharpbound{15} & 26 & 19
\\ \hline
9 &
\begin{tikzpicture}[scale=0.5]
\foreach \x in {0,1,2,3,4}{\draw (\x,0) node[vertex_su]{};}
\draw[thick] (0,0) to[bend left=50] (1,0);
\draw[thick] (0,0) to[bend left=50] (2,0);
\draw[thick] (0,0) to[bend left=50] (3,0);
\draw[thick] (3,0) to[bend left=50] (4,0);
\end{tikzpicture}
 & \sharpbound{13} & \sharpbound{13} & 21 & 16
\\ \hline
10 & 
\begin{tikzpicture}[scale=0.5]
\foreach \x in {0,1,2,3,4}{\draw (\x,0) node[vertex_su]{};}
\draw[thick] (0,0) to[bend left=50] (1,0);
\draw[thick] (0,0) to[bend left=50] (2,0);
\draw[thick] (0,0) to[bend left=50] (3,0);
\draw[thick] (2,0) to[bend left=50] (4,0);
\end{tikzpicture}
 & \sharpbound{12} & \sharpbound{12} & 22 & 16 \\ \bottomrule
\end{tabular}
\caption{Lower and upper bounds on the ordered Ramsey numbers of certain graphs of size $5$. ID is used to identify the graphs in the code repository.
}\label{fig:ordRam5}
\end{table}

\subsection{Implementation}

We used the programming language Julia to implement the lower bounds.
To compute the upper bounds based on flag algebras, we used the Julia package FlagSOS~\cite{Brosch2022} to generate the lists of Ramsey graphs up to isomorphism, which were then used by a further developed version of the flag algebra software developed by Lidick\'y and Pfender~\cite{Lidicky2021} to generate the list of flags, compute the products of flags, and build the final SDP.
The SDPs were solved using either Mosek~\cite{mosek} or CSDP~\cite{Borchers1999}, on an AMD EPYC $9474$F $48$-Core Processor with $1536$GB RAM, and on the Alderaan cluster at the University of Colorado Denver. The cluster was funded by an NSF grant~$2019089$ CC* Compute: Accelerating Science and Education by Campus and Grid Computing.
The ILP formulations were solved using Gurobi~\cite{gurobi} version $10.0.0$ with JuMP~\cite{Lubin2023} on an AMD EPYC $7532$ $32$-Core Processor with $1024$GB RAM.
Results of our calculations and programs to perform them are available at \url{https://github.com/DanielBrosch/Ramsey}.

\section{Conclusion}

By combining the flag algebra method, which uses colorblind and/or ordered flag algebras, with algorithms for constructing Ramsey graphs, we can efficiently compute numerous lower and upper bounds in an automated manner. We recovered some known Ramsey numbers but also obtained plenty of new bounds.
When the resulting bound is small, ILP methods are very efficient. 
For larger values, the combination of tabu search and flag algebras seemed to work better to obtain bounds.

\printbibliography

\newpage

\appendix 

\section{Appendix: Criteria on orderability of bipartite graphs}
\label{app:constraintOrderabilityBipartite}

Here, we detail the constraints we enforce in the colorblind ILP formulation~\eqref{eq:ORN_ILP} to forbid the presence of an orderable graph $G$ in an edge-coloring of $K_n$ when $G$ is one of the bipartite graphs $\Kb23, L_3, \Kbmin33, \Kb24$ and $\Kb33$. The conditions on the orderability of the graphs were established by considering all possible orderings of each graph and categorizing all orderable colorings into general cases.

For each edge-colored graph $G$, we detail the subgraphs that make $G$ orderable. We represent each of these subgraphs and the orderable $G$ it induces, drawn in its usual and ordered forms. We then express the linear constraint forbidding this subgraph. In each of the drawings, edge color classes are represented by different colors, though they may be combined, and black edges can be of any color.

 To simplify notation, we extend the vector of $y$-variables introduced in Section~\ref{ssec:ILPunorderedCanonical} to also include variables $y_{ab,cd}$ for $a=b$, for $c=d$ and for $\edge ab = \edge cd$, and we set all these variables to $0$.

\subsection{Complete bipartite graph $\Kb23$}

\begin{mdframed}
    \begin{figure}[H]
    \center
   \raisebox{1em}{\begin{tikzpicture}[scale=0.7]
    \draw[ultra thick, BrickRed](1,0) to (1,1);
    \draw[ultra thick, BrickRed](1,1) to (0,1.6);
    \draw[ultra thick, BrickRed] (1,1) to (2,1.6);
   \foreach \x/\y in {1/0, 1/1, 0/1.6, 2/1.6}{\draw (\x,\y) node[vertex_u]{};  }   
    \end{tikzpicture}}
    \hspace{1.5cm}
    \begin{tikzpicture}[scale=0.8]
  
    \draw[ultra thick, BrickRed] (0,1) to (1,-0.5);
    \draw[ultra thick, BrickRed] (0,1) to (1,0.5);
    \draw[ultra thick, BrickRed] (0,1) to (1,1.5);
    \draw[thick] (0,0) to (1,-0.5);
    \draw[thick] (0,0) to (1,0.5);
    \draw[thick] (0,0) to (1,1.5);
    \draw[color=black] (0,1) node[align=center,   left=1.2pt] {$a_1$};
    \draw[color=black] (0,0) node[align=center,   left=1.2pt] {$b_1$};
    \draw[color=black] (1,1.5) node[align=center,   right=1.2pt] {$a_2$};
    \draw[color=black] (1,0.5) node[align=center,   right=1.2pt] {$b_2$};
    \draw[color=black] (1,-0.5) node[align=center,   right=1.2pt] {$c_2$};
   \foreach \x/\y in {0/0, 0/1, 1/-0.5, 1/0.5, 1/1.5}{\draw (\x,\y) node[vertex_u]{};  }    
    \end{tikzpicture}
    
    \hspace{1.5cm}

    \begin{tikzpicture}[scale=0.85]
    \draw[ultra thick, BrickRed](0,0) to[bend left=25] (1,0);
    \draw[ultra thick, BrickRed](0,0) to[bend left=40] (2,0);
    \draw[ultra thick, BrickRed] (0,0) to[bend left=50] (3,0);
    \draw[thick] (4,0) to[bend right=50] (1,0);
    \draw[thick] (4,0) to[bend right=40] (2,0);
    \draw[thick] (4,0) to[bend right=25] (3,0);
    \draw[color=black] (0,0) node[align=center, text height=5pt,   below=1.2pt] {$a_1$};
    \draw[color=black] (1,0) node[align=center, text height=5pt,   below=1.2pt] {$a_2$};
    \draw[color=black] (2,0) node[align=center, text height=5pt,   below=1.2pt] {$b_2$};
    \draw[color=black] (3,0) node[align=center, text height=5pt,   below=1.2pt] {$c_2$};
    \draw[color=black] (4,0) node[align=center, text height=5pt,   below=1.2pt] {$b_1$};
    \foreach \x in {0,1,2,3,4}{\draw (\x,0) node[vertex_u]{};}    
    \end{tikzpicture}

    \end{figure}
    \vspace{-.5cm}

    \begin{equation*}
        \sum\limits_{b \notin\set{a_1,a_2} } y_{a_1a_2,a_1b}\leq 1 \qquad \forall a_1,a_2 \in \nint
    \end{equation*} 
    \end{mdframed}

\begin{mdframed}
    \begin{figure}[H]
    \center
    \raisebox{2em}{\begin{tikzpicture}[scale=0.7]
    \draw[ultra thick, BrickRed](0,0) to (1,0.6);
    \draw[ultra thick, BrickRed](0,0) to (1,-0.6);
    \draw[ultra thick, NavyBlue](2,0) to (1,0.6);
    \draw[ultra thick, NavyBlue](2,0) to (1,-0.6);
   \foreach \x/\y in {0/0, 1/0.6, 1/-0.6, 2/0}{\draw (\x,\y) node[vertex_u]{};  }         
    \end{tikzpicture}}
    \hspace{1.5cm}
    \begin{tikzpicture}[scale=0.8]
    \draw[thick] (0,1) to (1,-0.5);
    \draw[ultra thick, NavyBlue] (0,1) to (1,0.5);
    \draw[ultra thick, BrickRed] (0,1) to (1,1.5);
    \draw[thick] (0,0) to (1,-0.5);
    \draw[ultra thick, NavyBlue] (0,0) to (1,0.5);
    \draw[ultra thick, BrickRed] (0,0) to (1,1.5);
    \draw[color=black] (0,1) node[align=center,   left=1.2pt] {$a_1$};
    \draw[color=black] (0,0) node[align=center,   left=1.2pt] {$b_1$};
    \draw[color=black] (1,1.5) node[align=center,   right=1.2pt] {$a_2$};
    \draw[color=black] (1,0.5) node[align=center,   right=1.2pt] {$b_2$};
    \draw[color=black] (1,-0.5) node[align=center,   right=1.2pt] {$c_2$};
   \foreach \x/\y in {0/0, 0/1, 1/-0.5, 1/0.5, 1/1.5}{\draw (\x,\y) node[vertex_u]{};  }           
    \end{tikzpicture}

    \hspace{1.5cm}
    \begin{tikzpicture}[scale=.85]
    \draw[ultra thick, BrickRed](0,0) to[bend left=40] (2,0);
    \draw[ultra thick, BrickRed](0,0) to[bend left=50] (3,0);
    \draw[ultra thick, NavyBlue](1,0) to[bend left=25] (2,0);
    \draw[ultra thick, NavyBlue](1,0) to[bend left=40] (3,0);
    \draw[thick] (3,0) to[bend left=25] (4,0);
    \draw[thick] (2,0) to[bend left=40] (4,0);
    \draw[color=black] (0,0) node[align=center, text height=5pt,   below=1.2pt] {$a_2$};
    \draw[color=black] (1,0) node[align=center, text height=5pt,   below=1.2pt] {$b_2$};
    \draw[color=black] (2,0) node[align=center, text height=5pt,   below=1.2pt] {$a_1$};
    \draw[color=black] (3,0) node[align=center, text height=5pt,   below=1.2pt] {$b_1$};
    \draw[color=black] (4,0) node[align=center, text height=5pt,   below=1.2pt] {$c_2$};
    \foreach \x in {0,1,2,3,4}{\draw (\x,0) node[vertex_u]{};}        
    \end{tikzpicture}
    
    \end{figure}
    \vspace{-.8cm}
    
    \begin{align*}
        & y_{a_2a_1,a_2b_1} + y_{b_2a_1,b_2b_1} \leq 1 \qquad \forallpwd a_1,b_1, a_2, b_2 \in \nint
    \end{align*}
    \end{mdframed}

    \begin{mdframed}
        \begin{figure}[H]
        \center
        \raisebox{1.75em}{\begin{tikzpicture}[scale=0.7]
        \draw[ultra thick, BrickRed](0,0) to (1,0.5);
        \draw[ultra thick, BrickRed](0,0) to (1,-0.5);
        \draw[ultra thick, NavyBlue](1,0.5) to (2,1);
        \draw[ultra thick, NavyBlue](1,0.5) to (2,0);
   \foreach \x/\y in {0/0, 1/0.5, 1/-0.5, 2/0, 2/1}{\draw (\x,\y) node[vertex_u]{};  }                
        \end{tikzpicture}}
        \hspace{1.5cm}
        \begin{tikzpicture}[scale=0.8]
        \draw[ultra thick, NavyBlue] (0,1) to (1,-0.5);
        \draw[ultra thick, NavyBlue] (0,1) to (1,0.5);
        \draw[ultra thick, BrickRed] (0,1) to (1,1.5);
        \draw[thick] (0,0) to (1,-0.5);
        \draw[thick] (0,0) to (1,0.5);
        \draw[ultra thick, BrickRed] (0,0) to (1,1.5);
        \draw[color=black] (0,1) node[align=center,   left=1.2pt] {$a_1$};
        \draw[color=black] (0,0) node[align=center,   left=1.2pt] {$b_1$};
        \draw[color=black] (1,1.5) node[align=center,   right=1.2pt] {$a_2$};
        \draw[color=black] (1,0.5) node[align=center,   right=1.2pt] {$b_2$};
        \draw[color=black] (1,-0.5) node[align=center,   right=1.2pt] {$c_2$};
   \foreach \x/\y in {0/0, 0/1, 1/-0.5, 1/0.5, 1/1.5}{\draw (\x,\y) node[vertex_u]{};  }                   
        \end{tikzpicture}

        \hspace{1.5cm}

        \begin{tikzpicture}[scale=.85]
        \draw[ultra thick, BrickRed](0,0) to[bend left=25] (1,0);
        \draw[ultra thick, BrickRed](0,0) to[bend left=40] (2,0);
        \draw[ultra thick, NavyBlue](1,0) to[bend left=40] (3,0);
        \draw[ultra thick, NavyBlue](1,0) to[bend left=50] (4,0);
        \draw[thick] (2,0) to[bend left=25] (3,0);
        \draw[thick] (2,0) to[bend left=40] (4,0);
        \draw[color=black] (0,0) node[align=center, text height=5pt,  below=1.2pt] {$a_2$};
        \draw[color=black] (1,0) node[align=center, text height=5pt,  below=1.2pt] {$a_1$};
        \draw[color=black] (2,0) node[align=center, text height=5pt,  below=1.2pt] {$b_1$};
        \draw[color=black] (3,0) node[align=center, text height=5pt,  below=1.2pt] {$b_2$};
        \draw[color=black] (4,0) node[align=center, text height=5pt,  below=1.2pt] {$c_2$};
        \foreach \x in {0,1,2,3,4}{\draw (\x,0) node[vertex_u]{};}        
        \end{tikzpicture}

        \end{figure}
        
        \vspace{-.8cm}
        \begin{align*}
            &  y_{a_2a_1,a_2b_1} + y_{a_1b_2,a_1c_2} \leq 1 \qquad \forallpwd a_1,b_1,a_2, b_2,c_2 \in \nint
        \end{align*}
        \end{mdframed}

\subsection{Ladder graph $L_3$}

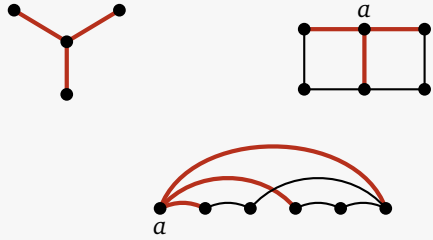
\vspace{1cm}

\begin{mdframed}
    \begin{figure}[H]
    \center
    \hspace{.5cm}
    \raisebox{.8em}{\begin{tikzpicture}[scale=0.7]
    \draw[ultra thick, BrickRed](1,0) to (1,1);
    \draw[ultra thick, BrickRed](1,1) to (0,1.6);
    \draw[ultra thick, BrickRed] (1,1) to (2,1.6);
   \foreach \x/\y in {1/0, 1/1, 0/1.6, 2/1.6}{\draw (\x,\y) node[vertex_u]{};  }                
    \end{tikzpicture}}
    \hspace{2.1cm}
    \raisebox{1em}{\begin{tikzpicture}[scale=0.8]
    \draw[thick] (0,0) to (1,0);
    \draw[ultra thick, BrickRed] (1,0) to (1,1);
    \draw[ultra thick, BrickRed] (1,1) to (0,1);
    \draw[thick] (0,1) to (0,0);
    \draw[thick] (1,0) to (2,0);
    \draw[thick] (2,0) to (2,1);
    \draw[ultra thick, BrickRed] (2,1) to (1,1);

    \draw[color=black] (1,1) node[align=center,   above=1.2pt] {$a$};

   \foreach \x/\y in {0/0, 1/0, 1/1, 0/1, 2/1,2/0}{\draw (\x,\y) node[vertex_u]{};  }                
    \end{tikzpicture}}
    
    \hspace{1.8cm}
    \begin{tikzpicture}[scale=.6]
    \draw[ultra thick, BrickRed](0,0) to[bend left=25] (1,0);
    \draw[ultra thick, BrickRed](0,0) to[bend left=50] (3,0);
    \draw[ultra thick, BrickRed](0,0) to[bend left=70] (5,0);
    \draw[thick] (1,0) to[bend left=25] (2,0);
    \draw[thick] (2,0) to[bend left=50] (5,0);
    \draw[thick] (3,0) to[bend left=25] (4,0);
    \draw[thick] (4,0) to[bend left=25] (5,0);
    \draw[color=black] (0,0) node[align=center, text height=5pt,   below=1.2pt] {$a$};

    \foreach \x in {0,1,2,3,4,5}{\draw (\x,0) node[vertex_u]{};}    
    \end{tikzpicture}
    
    \end{figure}
    
    \vspace{-.6cm}
    \begin{equation*}
        \sum\limits_{c \notin\set{a,b} } y_{ab,ac}\leq 1 \qquad \forall a,b \in \nint
    \end{equation*} 
    \end{mdframed}

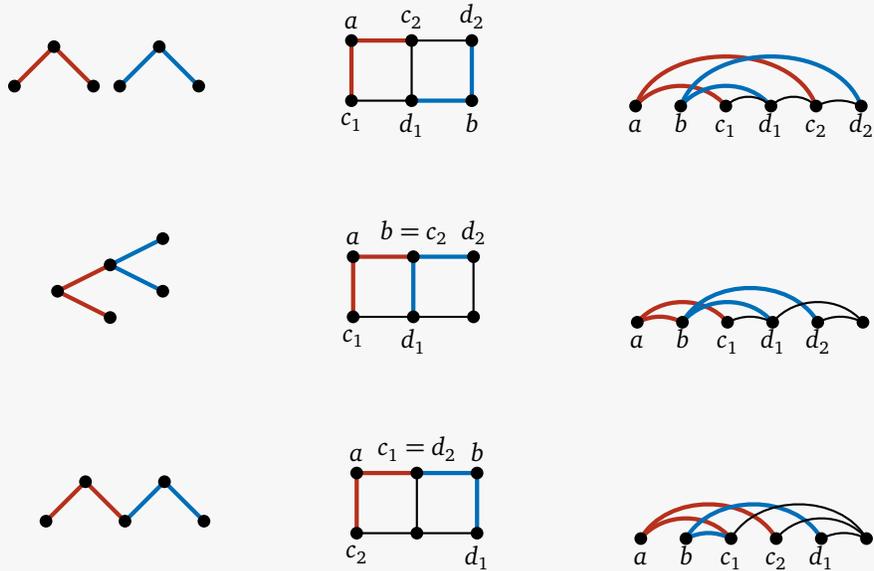
\vspace{.5cm}

\begin{mdframed}
    \begin{figure}[H]
    \center
    \hspace{.15cm}
    \raisebox{2.1em}{\begin{tikzpicture}[scale=0.7]
    \draw[ultra thick, BrickRed](0,0) to (0.75,0.75);
    \draw[ultra thick, BrickRed](0.75,0.75) to (1.5,0);
    \draw[ultra thick, NavyBlue](2,0) to (2.75,0.75);
    \draw[ultra thick, NavyBlue](2.75,0.75) to (3.5,0);
   \foreach \x/\y in {0/0,  0.75/0.75, 1.5/0, 2.75/0.75, 3.5/0, 2/0}{\draw (\x,\y) node[vertex_u]{};  }     
    \end{tikzpicture}}
    \hspace{1.5cm}
    \begin{tikzpicture}[scale=0.8]
    \draw[thick] (0,0) to (1,0);
    \draw[thick] (1,0) to (1,1);
    \draw[ultra thick, BrickRed] (1,1) to (0,1);
    \draw[ultra thick, BrickRed] (0,1) to (0,0);
    \draw[ultra thick, NavyBlue] (1,0) to (2,0);
    \draw[ultra thick, NavyBlue] (2,0) to (2,1);
    \draw[thick] (2,1) to (1,1);

    \draw[color=black] (0,1) node[align=center,   above=1.2pt] {$a$};
    \draw[color=black] (0,0) node[align=center,   below=1.2pt] {$c_1$};
    \draw[color=black] (1,1) node[align=center,   above=1.2pt] {$c_2$};
    \draw[color=black] (1,0) node[align=center,   below=1.2pt] {$d_1$};
    \draw[color=black] (2,1) node[align=center,   above=1.2pt] {$d_2$};
    \draw[color=black] (2,0) node[align=center,   below=1.2pt] {$b$};
   \foreach \x/\y in {0/0, 1/0, 1/1, 0/1, 2/1,2/0}{\draw (\x,\y) node[vertex_u]{};  }                
    \end{tikzpicture}
    \hspace{1.5cm}
    \begin{tikzpicture}[scale=.6]
    \draw[ultra thick, BrickRed](0,0) to[bend left=50] (2,0);
    \draw[ultra thick, BrickRed](0,0) to[bend left=70] (4,0);
    \draw[ultra thick, NavyBlue](1,0) to[bend left=50] (3,0);
    \draw[ultra thick, NavyBlue](1,0) to[bend left=70] (5,0);
    \draw[thick] (2,0) to[bend left=45] (3,0);
    \draw[thick] (3,0) to[bend left=45] (4,0);
    \draw[thick] (4,0) to[bend left=25] (5,0);
    \draw[color=black] (0,0) node[align=center, text height=5pt,   below=1.2pt] {$a$};
    \draw[color=black] (1,0) node[align=center, text height=5pt,   below=1.2pt] {$b$};
    \draw[color=black] (2,0) node[align=center, text height=5pt,   below=1.2pt] {$c_1$};
    \draw[color=black] (3,0) node[align=center, text height=5pt,   below=1.2pt] {$d_1$};
    \draw[color=black] (4,0) node[align=center, text height=5pt,   below=1.2pt] {$c_2$};
    \draw[color=black] (5,0) node[align=center, text height=5pt,   below=1.2pt] {$d_2$};
    \foreach \x in {0,1,2,3,4,5}{\draw (\x,0) node[vertex_u]{};}      
    \end{tikzpicture}
    
    \end{figure}
    
    \begin{figure}[H]
    \center
    \hspace{.55cm}
    \raisebox{1.5em}{\begin{tikzpicture}[scale=0.7]
    \draw[ultra thick, BrickRed](0,0) to (1,0.5);
    \draw[ultra thick, BrickRed](0,0) to (1,-0.5);
    \draw[ultra thick, NavyBlue](1,0.5) to (2,1);
    \draw[ultra thick, NavyBlue](1,0.5) to (2,0);
   \foreach \x/\y in {0/0, 1/0.5, 1/-0.5, 2/0, 2/1}{\draw (\x,\y) node[vertex_u]{};  }            
    \end{tikzpicture}}
    \hspace{2.cm}
    \begin{tikzpicture}[scale=0.8]
    \draw[thick] (0,0) to (1,0);
    \draw[ultra thick, NavyBlue](1,0) to (1,1);
    \draw[ultra thick, BrickRed] (1,1) to (0,1);
    \draw[ultra thick, BrickRed] (0,1) to (0,0);
    \draw[thick] (1,0) to (2,0);
    \draw[thick] (2,0) to (2,1);
    \draw[ultra thick, NavyBlue] (2,1) to (1,1);

    \draw[color=black] (0,1) node[align=center,   above=1.2pt] {$a$};
    \draw[color=black] (0,0) node[align=center,   below=1.2pt] {$c_1$};
    \draw[color=black] (1,1) node[align=center,   above=1.2pt] {$b=c_2$};
    \draw[color=black] (1,0) node[align=center,   below=1.2pt] {$d_1$};
    \draw[color=black] (2,1) node[align=center,   above=1.2pt] {$d_2$};
   \foreach \x/\y in {0/0, 1/0, 1/1, 0/1, 2/1,2/0}{\draw (\x,\y) node[vertex_u]{};  }                
    \end{tikzpicture}
    \hspace{1.5cm}
    \begin{tikzpicture}[scale=.6]
    \draw[ultra thick, BrickRed](0,0) to[bend left=25] (1,0);
    \draw[ultra thick, BrickRed](0,0) to[bend left=50] (2,0);
    \draw[ultra thick, NavyBlue](1,0) to[bend left=50] (3,0);
    \draw[ultra thick, NavyBlue](1,0) to[bend left=60] (4,0);
    \draw[thick] (2,0) to[bend left=25] (3,0);
    \draw[thick] (3,0) to[bend left=50] (5,0);
    \draw[thick] (4,0) to[bend left=25] (5,0);
    \draw[color=black] (0,0) node[align=center, text height=5pt,   below=1.2pt] {$a$};
    \draw[color=black] (1,0) node[align=center, text height=5pt,   below=1.2pt] {$b$};
    \draw[color=black] (2,0) node[align=center, text height=5pt,   below=1.2pt] {$c_1$};
    \draw[color=black] (3,0) node[align=center, text height=5pt,   below=1.2pt] {$d_1$};
    
    \draw[color=black] (4,0) node[align=center, text height=5pt,   below=1.2pt] {$d_2$};
    \foreach \x in {0,1,2,3,4,5}{\draw (\x,0) node[vertex_u]{};}     
    \end{tikzpicture}
    
    \end{figure}
    
    \begin{figure}[H]
    \center
    \hspace{.35cm}
    \raisebox{2em}{
    \begin{tikzpicture}[scale=0.7]
    \draw[ultra thick, BrickRed](0,0) to (0.75,0.75);
    \draw[ultra thick, BrickRed](0.75,0.75) to (1.5,0);
    \draw[ultra thick, NavyBlue](1.5,0) to (2.25,0.75);
    \draw[ultra thick, NavyBlue](2.25,0.75) to (3,0);
   \foreach \x/\y in {0/0,  0.75/0.75, 1.5/0, 2.25/0.75, 3/0}{\draw (\x,\y) node[vertex_u]{};  }                
    \end{tikzpicture}}
    \hspace{1.5cm}
    \begin{tikzpicture}[scale=0.8]
    \draw[thick] (0,0) to (1,0);
    \draw[thick](1,0) to (1,1);
    \draw[ultra thick, BrickRed] (1,1) to (0,1);
    \draw[ultra thick, BrickRed] (0,1) to (0,0);
    \draw[thick] (1,0) to (2,0);
    \draw[ultra thick, NavyBlue] (2,0) to (2,1);
    \draw[ultra thick, NavyBlue] (2,1) to (1,1);
    \draw[color=black] (0,1) node[align=center,   above=1.2pt] {$a$};
    \draw[color=black] (0,0) node[align=center,   below=1.2pt] {$c_2$};
    \draw[color=black] (1,1) node[align=center,   above=1.2pt] {$c_1=d_2$};
    \draw[color=black] (2,1) node[align=center,   above=1.2pt] {$b$};
    \draw[color=black] (2,0) node[align=center,   below=1.2pt] {$d_1$};
   \foreach \x/\y in {0/0, 1/0, 1/1, 0/1, 2/1,2/0}{\draw (\x,\y) node[vertex_u]{};  }                
    \end{tikzpicture}
    \hspace{1.5cm}
    \begin{tikzpicture}[scale=.6]
    \draw[ultra thick, BrickRed](0,0) to[bend left=60] (3,0);
    \draw[ultra thick, BrickRed](0,0) to[bend left=50] (2,0);
    \draw[ultra thick, NavyBlue](1,0) to[bend left=25] (2,0);
    \draw[ultra thick, NavyBlue](1,0) to[bend left=60] (4,0);
    \draw[thick] (2,0) to[bend left=60] (5,0);
    \draw[thick] (3,0) to[bend left=50] (5,0);
    \draw[thick] (4,0) to[bend left=25] (5,0);
    \draw[color=black] (0,0) node[align=center, text height=5pt,   below=1.2pt] {$a$};
    \draw[color=black] (1,0) node[align=center, text height=5pt,   below=1.2pt] {$b$};
    \draw[color=black] (2,0) node[align=center, text height=5pt,   below=1.2pt] {$c_1$};
    \draw[color=black] (3,0) node[align=center, text height=5pt,   below=1.2pt] {$c_2$};
    
    \draw[color=black] (4,0) node[align=center, text height=5pt,   below=1.2pt] {$d_1$};
    \foreach \x in {0,1,2,3,4,5}{\draw (\x,0) node[vertex_u]{};}     
    \end{tikzpicture}

    \end{figure}
     
     \vspace{-.6cm}
     \begin{equation*}
         y_{ac_1,ac_2} + y_{bd_1,bd_2} \leq 1 \qquad \forall a,b,c_1,c_2,d_1,d_2 \in \nint \text{ with } \abs{\set{a,b,c_1,c_2, d_1, d_2}} \geq 5 \\
     \end{equation*}
     \end{mdframed}

\newpage
\subsection{Complete bipartite graph minus one edge $\Kbmin{3}{3}$}

\begin{mdframed}
\begin{figure}[H]
\center
\raisebox{1.9em}{
\begin{tikzpicture}[scale=0.7]
\draw[ultra thick, BrickRed](-0.85,0.65) to (0,0);
\draw[ultra thick, BrickRed](-0.85,-0.65) to (0,0);
\draw[ultra thick, BrickRed](0,0) to (1,0);
\draw[ultra thick, NavyBlue](1,0) to (1.85,0.65);
\draw[ultra thick, NavyBlue](1,0) to (1.85,-0.65);
 \foreach \x/\y in {0/0, -0.85/0.65, -0.85/-0.65, 1/0, 1.85/0.65,1.85/-0.65}{\draw (\x,\y) node[vertex_u]{};  }    
\end{tikzpicture}}
\hspace{1.5cm}
\begin{tikzpicture}[scale=0.8]
\draw[ultra thick, BrickRed]  (0,2) to (1,2);
\draw[ultra thick, BrickRed]  (0,2) to (1,1);
\draw[ultra thick, BrickRed]  (0,2) to (1,0);
\draw[ultra thick, NavyBlue] (0,1) to (1,2);
\draw[thick] (0,1) to (1,1);
\draw[thick] (0,1) to (1,0);
\draw[ultra thick, NavyBlue] (0,0) to (1,2);
\draw[thick] (0,0) to (1,1);
\draw[color=black] (0,2) node[align=center,   left=1.2pt] {$a_1$};
\draw[color=black] (0,1) node[align=center,   left=1.2pt] {$b_1$};
\draw[color=black] (0,0) node[align=center,   left=1.2pt] {$c_1$};
\draw[color=black] (1,2) node[align=center,   right=1.2pt] {$a_2$};
\draw[color=black] (1,1) node[align=center,   right=1.2pt] {$b_2$};
\draw[color=black] (1,0) node[align=center,   right=1.2pt] {$c_2$};

 \foreach \x/\y in {0/0, 0/1, 0/2, 1/0, 1/1, 1/2}{\draw (\x,\y) node[vertex_u]{};  }    
\end{tikzpicture}
\hspace{1.5cm}
\begin{tikzpicture}[scale=.6]
\draw[ultra thick, BrickRed](0,0) to[bend left=25] (1,0);
\draw[ultra thick, BrickRed](0,0) to[bend left=50] (3,0);
\draw[ultra thick, BrickRed](0,0) to[bend left=50] (5,0);
\draw[ultra thick, NavyBlue](1,0) to[bend left=25] (2,0);
\draw[ultra thick, NavyBlue](1,0) to[bend left=50] (4,0);
\draw[thick] (2,0) to[bend left=25] (3,0);
\draw[thick] (3,0) to[bend left=25] (4,0);
\draw[thick] (4,0) to[bend left=25] (5,0);
\draw[color=black] (0,0) node[align=center, text height=5pt,   below=1.2pt] {$a_1$};
\draw[color=black] (1,0) node[align=center, text height=5pt,   below=1.2pt] {$a_2$};
\draw[color=black] (2,0) node[align=center, text height=5pt,   below=1.2pt] {$c_1$};
\draw[color=black] (3,0) node[align=center, text height=5pt,   below=1.2pt] {$b_2$};
\draw[color=black] (4,0) node[align=center, text height=5pt,   below=1.2pt] {$b_1$};
\draw[color=black] (5,0) node[align=center, text height=5pt,   below=1.2pt] {$c_2$};
    \foreach \x in {0,1,2,3,4,5}{\draw (\x,0) node[vertex_u]{};}  
\end{tikzpicture}

\end{figure}

\vspace{-.75cm}
\begin{equation*}
    y_{a_1a_2,a_1b_2} + y_{a_1a_2,a_1c_2} + y_{a_1b_2,a_1c_2} +  y_{a_2b_1,a_2c_1} \leq 3 \qquad \forallpwd a_1, b_1, c_1, a_2, b_2, c_2 \in \nint
\end{equation*}
\end{mdframed}

\begin{mdframed}
\begin{figure}[H]
\center
\raisebox{2em}{
\begin{tikzpicture}[scale=0.7]
\draw[ultra thick, BrickRed](-1,0) to (0,0);
\draw[ultra thick, BrickRed](0,0) to (1,0.6);
\draw[ultra thick, BrickRed](0,0) to (1,-0.6);
\draw[ultra thick, NavyBlue](2,0) to (1,0.6);
\draw[ultra thick, NavyBlue](2,0) to (1,-0.6);
 \foreach \x/\y in {0/0, -1/0, 1/0.6, 1/-0.6, 2/0}{\draw (\x,\y) node[vertex_u]{};  } 
\end{tikzpicture}}
\hspace{1.5cm}
\begin{tikzpicture}[scale=0.8]
\draw[ultra thick, BrickRed]  (0,2) to (1,2);
\draw[ultra thick, BrickRed]  (0,2) to (1,1);
\draw[ultra thick, BrickRed]  (0,2) to (1,0);
\draw[thick] (0,1) to (1,2);
\draw[thick] (0,1) to (1,1);
\draw[thick] (0,1) to (1,0);
\draw[ultra thick, NavyBlue] (0,0) to (1,2);
\draw[ultra thick, NavyBlue] (0,0) to (1,1);
\draw[color=black] (0,2) node[align=center,   left=1.2pt] {$a_1$};
\draw[color=black] (0,0) node[align=center,   left=1.2pt] {$c_1$};
\draw[color=black] (0,1) node[align=center,   left=1.2pt] {$b_1$};
\draw[color=black] (1,2) node[align=center,   right=1.2pt] {$a_2$};
\draw[color=black] (1,1) node[align=center,   right=1.2pt] {$b_2$};
\draw[color=black] (1,0) node[align=center,   right=1.2pt] {$c_2$};

 \foreach \x/\y in {0/0, 0/1, 0/2, 1/0, 1/1, 1/2}{\draw (\x,\y) node[vertex_u]{};  }    
\end{tikzpicture}
\hspace{1.5cm}
\begin{tikzpicture}[scale=.6]
\draw[ultra thick, BrickRed](0,0) to[bend left=40] (2,0);
\draw[ultra thick, BrickRed](0,0) to[bend left=50] (3,0);
\draw[ultra thick, BrickRed](0,0) to[bend left=60] (4,0);
\draw[ultra thick, NavyBlue](1,0) to[bend left=25] (2,0);
\draw[ultra thick, NavyBlue](1,0) to[bend left=40] (3,0);
\draw[thick] (2,0) to[bend left=50] (5,0);
\draw[thick] (3,0) to[bend left=40] (5,0);
\draw[thick] (4,0) to[bend left=25] (5,0);
\draw[color=black] (0,0) node[align=center, text height=5pt,   below=1.2pt] {$a_1$};
\draw[color=black] (1,0) node[align=center, text height=5pt,   below=1.2pt] {$c_1$};
\draw[color=black] (2,0) node[align=center, text height=5pt,   below=1.2pt] {$a_2$};
\draw[color=black] (3,0) node[align=center, text height=5pt,   below=1.2pt] {$b_2$};
\draw[color=black] (4,0) node[align=center, text height=5pt,   below=1.2pt] {$c_2$};
\draw[color=black] (5,0) node[align=center, text height=5pt,   below=1.2pt] {$b_1$};
    \foreach \x in {0,1,2,3,4,5}{\draw (\x,0) node[vertex_u]{};}  
\end{tikzpicture}

\end{figure}

\vspace{-.75cm}
\begin{equation*}
y_{a_1a_2,a_1b_2} + y_{a_1a_2,a_1c_2} + y_{a_1b_2,a_1c_2} + y_{c_1a_2, c_1b_2} \leq 3 \qquad \forallpwd a_1, b_1, c_1, a_2, b_2, c_2 \in \nint
\end{equation*}
\end{mdframed}

\begin{mdframed}
\begin{figure}[H]
\center
\raisebox{2em}{
\begin{tikzpicture}[scale=0.7]        
\draw[ultra thick, BrickRed](0,0) to (1,0.5);
\draw[ultra thick, BrickRed](0,0) to (1,-0.5);
\draw[ultra thick, NavyBlue](1,0.5) to (2,1);
\draw[ultra thick, NavyBlue](1,0.5) to (2,0);
\draw[ultra thick, Green](1,-0.5) to (2,0);
\draw[ultra thick, Green](2,0) to (3,-0.5);
   \foreach \x/\y in {0/0, 1/0.5, 1/-0.5, 2/0, 2/1,3/-0.5}{\draw (\x,\y) node[vertex_u]{};  }    
\end{tikzpicture}}
\hspace{1.5cm}
\begin{tikzpicture}[scale=0.8]
\draw[ultra thick, NavyBlue]  (0,2) to (1,2);
\draw[ultra thick, Green]  (0,2) to (1,1);
\draw[ultra thick, Green]  (0,2) to (1,0);
\draw[ultra thick, NavyBlue] (0,1) to (1,2);
\draw[thick] (0,1) to (1,1);
\draw[thick] (0,1) to (1,0);
\draw[ultra thick, BrickRed] (0,0) to (1,2);
\draw[ultra thick, BrickRed] (0,0) to (1,1);
\draw[color=black] (0,2) node[align=center,   left=1.2pt] {$a_1$};
\draw[color=black] (0,1) node[align=center,   left=1.2pt] {$b_1$};
\draw[color=black] (0,0) node[align=center,   left=1.2pt] {$c_1$};
\draw[color=black] (1,2) node[align=center,   right=1.2pt] {$a_2$};
\draw[color=black] (1,1) node[align=center,   right=1.2pt] {$b_2$};
\draw[color=black] (1,0) node[align=center,   right=1.2pt] {$c_2$};

 \foreach \x/\y in {0/0, 0/1, 0/2, 1/0, 1/1, 1/2}{\draw (\x,\y) node[vertex_u]{};  }    
\end{tikzpicture}
\hspace{1.5cm}
\begin{tikzpicture}[scale=.6]
\draw[ultra thick, BrickRed](0,0) to[bend left=25] (1,0);
\draw[ultra thick, BrickRed](0,0) to[bend left=60] (3,0);
\draw[ultra thick, NavyBlue](1,0) to[bend left=25] (2,0);
\draw[ultra thick, NavyBlue](1,0) to[bend left=60] (4,0);
\draw[ultra thick, Green] (2,0) to[bend left=25] (3,0);
\draw[ultra thick, Green] (2,0) to[bend left=60] (5,0);
\draw[thick] (3,0) to[bend left=25] (4,0);
\draw[thick] (4,0) to[bend left=25] (5,0);
\draw[color=black] (0,0) node[align=center,  text height=5pt,   below=1.2pt] {$c_1$};
\draw[color=black] (1,0) node[align=center,  text height=5pt,   below=1.2pt] {$a_2$};
\draw[color=black] (2,0) node[align=center,  text height=5pt,   below=1.2pt] {$a_1$};
\draw[color=black] (3,0) node[align=center,  text height=5pt,   below=1.2pt] {$b_2$};
\draw[color=black] (4,0) node[align=center,  text height=5pt,   below=1.2pt] {$b_1$};
\draw[color=black] (5,0) node[align=center,  text height=5pt,   below=1.2pt] {$c_2$};
    \foreach \x in {0,1,2,3,4,5}{\draw (\x,0) node[vertex_u]{};}  
\end{tikzpicture}

\end{figure}

\vspace{-.75cm}
  \begin{align*}
     y_{a_1b_2,a_1c_2}  + y_{a_2a_1,a_2b_1} + y_{c_1a_2, c_1b_2}  \leq 2 \qquad \forallpwd a_1, b_1, c_1, a_2, b_2, c_2 \in \nint
 \end{align*}
\end{mdframed}
\begin{mdframed}
\begin{figure}[H]
\center
\raisebox{2em}{
\begin{tikzpicture}[scale=0.7]
\draw[ultra thick, BrickRed](0,0) to (1,0.5);
\draw[ultra thick, BrickRed](0,0) to (1,-0.5);
\draw[ultra thick, NavyBlue](1,0.5) to (2,1);
\draw[ultra thick, NavyBlue](1,0.5) to (2,0);
\draw[ultra thick, Green](2,1) to (3,0.5);
\draw[ultra thick, Green](2,0) to (3,0.5);
   \foreach \x/\y in {0/0, 1/0.5, 1/-0.5, 2/0, 2/1,3/0.5}{\draw (\x,\y) node[vertex_u]{};  }   
\end{tikzpicture}}
\hspace{1.5cm}
\begin{tikzpicture}[scale=0.8]
\draw[ultra thick, NavyBlue]  (0,2) to (1,2);
\draw[thick]  (0,2) to (1,1);
\draw[ultra thick, Green]  (0,2) to (1,0);
\draw[ultra thick, NavyBlue] (0,1) to (1,2);
\draw[thick] (0,1) to (1,1);
\draw[ultra thick, Green] (0,1) to (1,0);
\draw[ultra thick, BrickRed] (0,0) to (1,2);
\draw[ultra thick, BrickRed] (0,0) to (1,1);
\draw[color=black] (0,2) node[align=center,   left=1.2pt] {$a_1$};
\draw[color=black] (0,1) node[align=center,   left=1.2pt] {$b_1$};
\draw[color=black] (0,0) node[align=center,   left=1.2pt] {$c_1$};
\draw[color=black] (1,2) node[align=center,   right=1.2pt] {$a_2$};
\draw[color=black] (1,1) node[align=center,   right=1.2pt] {$b_2$};
\draw[color=black] (1,0) node[align=center,   right=1.2pt] {$c_2$};

 \foreach \x/\y in {0/0, 0/1, 0/2, 1/0, 1/1, 1/2}{\draw (\x,\y) node[vertex_u]{};  }    
\end{tikzpicture}
\hspace{1.5cm}

\begin{tikzpicture}[scale=.6]
\draw[ultra thick, BrickRed](0,0) to[bend left=25] (1,0);
\draw[ultra thick, BrickRed](0,0) to[bend left=60] (5,0);
\draw[ultra thick, NavyBlue](1,0) to[bend left=50] (3,0);
\draw[ultra thick, NavyBlue](1,0) to[bend left=60] (4,0);
\draw[ultra thick, Green] (2,0) to[bend left=25] (3,0);
\draw[ultra thick, Green] (2,0) to[bend left=50] (4,0);
\draw[thick] (3,0) to[bend left=50] (5,0);
\draw[thick] (4,0) to[bend left=25] (5,0);
\draw[color=black] (0,0) node[align=center, text height=5pt,   below=1.2pt] {$c_1$};
\draw[color=black] (1,0) node[align=center, text height=5pt,   below=1.2pt] {$a_2$};
\draw[color=black] (2,0) node[align=center, text height=5pt,   below=1.2pt] {$c_2$};
\draw[color=black] (3,0) node[align=center, text height=5pt,   below=1.2pt] {$a_1$};
\draw[color=black] (4,0) node[align=center, text height=5pt,   below=1.2pt] {$b_1$};
\draw[color=black] (5,0) node[align=center, text height=5pt,   below=1.2pt] {$b_2$};
    \foreach \x in {0,1,2,3,4,5}{\draw (\x,0) node[vertex_u]{};}  
\end{tikzpicture}

\end{figure}

\vspace{-.75cm}
  \begin{align*}
      y_{a_2a_1,a_2b_1} + y_{c_1a_2, c_1b_2} +  y_{c_2a_1,c_2b_1} \leq 2 \qquad \forallpwd a_1, b_1, c_1, a_2, b_2, c_2 \in \nint
 \end{align*}
 \end{mdframed}

\newpage
\subsection{Complete bipartite graph $\Kb{2}{4}$}

\begin{mdframed}
\begin{figure}[H]
\center
\raisebox{.75em}{
\hspace{.25cm}
\begin{tikzpicture}[scale=0.7]
\draw[ultra thick, BrickRed](0,0) to (1,0.25);
\draw[ultra thick, BrickRed](0,0) to (1,0.8);
\draw[ultra thick, BrickRed](0,0) to (1,-0.25);
\draw[ultra thick, BrickRed](0,0) to (1,-0.8);
 \foreach \x/\y in {0/0, 1/0.25, 1/0.8, 1/-0.8, 1/-0.25}{\draw (\x,\y) node[vertex_u]{};  }     

\end{tikzpicture}}
\hspace{2.1cm}
\begin{tikzpicture}[scale=0.8]
\draw[ultra thick, BrickRed] (0,2) to (1,2.5);
\draw[ultra thick, BrickRed] (0,2) to (1,1.85);
\draw[ultra thick, BrickRed] (0,2) to (1,1.15);
\draw[ultra thick, BrickRed] (0,2) to (1,0.5);
\draw[thick] (0,1) to (1,2.5);
\draw[thick] (0,1) to (1,1.85);
\draw[thick] (0,1) to (1,1.15);
\draw[thick] (0,1) to (1,0.5);
\draw[color=black] (0,2) node[align=center,   left=1.2pt] {$a$};
\draw[color=black] (1,2.5) node[align=center,   left=1.2pt] {};
 \foreach \x/\y in {0/1, 0/2, 1/0.5, 1/1.85, 1/1.15, 1/2.5}{\draw (\x,\y) node[vertex_u]{};  }     
\end{tikzpicture}
\hspace{2.25cm}
\begin{tikzpicture}[scale=.6]
    \foreach \x in {0,1,2,3,4,5}{\draw (\x,0) node[vertex_u]{};}  
\draw[ultra thick, BrickRed](0,0) to[bend left=25] (1,0);
\draw[ultra thick, BrickRed](0,0) to[bend left=40] (2,0);
\draw[ultra thick, BrickRed](0,0) to[bend left=50] (3,0);
\draw[ultra thick, BrickRed](0,0) to[bend left=60] (4,0);
\draw[thick] (1,0) to[bend left=60] (5,0);
\draw[thick] (2,0) to[bend left=50] (5,0);
\draw[thick] (3,0) to[bend left=40] (5,0);
\draw[thick] (4,0) to[bend left=25] (5,0);
\draw[color=black] (0,0) node[align=center, text height=5pt, text height=5pt, text height=5pt,   below=1.2pt] {$a$};

\end{tikzpicture}

\end{figure}

\vspace{-.4cm}
\begin{equation}
\label{eq:K23_no4edges}
    \sum\limits_{c \notin\set{a,b}}y_{ab,ac} \leq 2 \qquad\foralld a,b \in \nint
\end{equation}
\end{mdframed}

\begin{mdframed}
    \begin{figure}[H]
    \center
    \raisebox{1em}{
    \begin{tikzpicture}[scale=0.7]
    \draw[ultra thick, BrickRed](0,0) to (1,0.5);
    \draw[ultra thick, BrickRed](0,0) to (1,-0.5);
    \draw[ultra thick, NavyBlue](1,0.5) to (2,1.25);
    \draw[ultra thick, NavyBlue](1,0.5) to (2.1,0.5);
    \draw[ultra thick, NavyBlue](1,0.5) to (2,-0.25);
 \foreach \x/\y in {0/0, 1/0.5, 1/-0.5, 2.1/0.5, 2/1.25, 2/-0.25}{\draw (\x,\y) node[vertex_u]{};  }     
    
    \end{tikzpicture}}
    \hspace{1.5cm}
    \begin{tikzpicture}[scale=0.8]
    \draw[ultra thick, BrickRed] (0,2) to (1,2.5);
    \draw[ultra thick, NavyBlue] (0,2) to (1,1.85);
    \draw[ultra thick, NavyBlue] (0,2) to (1,1.15);
    \draw[ultra thick, NavyBlue] (0,2) to (1,0.5);
    \draw[ultra thick, BrickRed] (0,1) to (1,2.5);
    \draw[thick] (0,1) to (1,1.85);
    \draw[thick] (0,1) to (1,1.15);
    \draw[thick] (0,1) to (1,0.5);
    \draw[color=black] (0,2) node[align=center,   left=1.2pt] {$a_1$};
    \draw[color=black] (1,2.5) node[align=center,   right=1.2pt] {$a_2$};
    \draw[color=black] (0,1) node[align=center,   left=1.2pt] {$b_1$};
    \draw[color=black] (1,1.85) node[align=center,   right=1.2pt] {$b_2$};
    \draw[color=black] (1,1.15) node[align=center,   right=1.2pt] {$c_2$};
    \draw[color=black] (1,0.5) node[align=center,   right=1.2pt] {$d_2$};
 \foreach \x/\y in {0/1, 0/2, 1/0.5, 1/1.85, 1/1.15, 1/2.5}{\draw (\x,\y) node[vertex_u]{};  }         
    \end{tikzpicture}
    \hspace{1.5cm}
    \begin{tikzpicture}[scale=.6]
    \foreach \x in {0,1,2,3,4,5}{\draw (\x,0) node[vertex_u]{};}      
    \draw[ultra thick, BrickRed](0,0) to[bend left=25] (1,0);
    \draw[ultra thick, BrickRed](0,0) to[bend left=60] (5,0);
    \draw[ultra thick, NavyBlue](1,0) to[bend left=40] (2,0);
    \draw[ultra thick, NavyBlue](1,0) to[bend left=50] (3,0);
    \draw[ultra thick, NavyBlue](1,0) to[bend left=60] (4,0);
    \draw[thick] (2,0) to[bend left=60] (5,0);
    \draw[thick] (3,0) to[bend left=50] (5,0);
    \draw[thick] (4,0) to[bend left=40] (5,0);
    \draw[color=black] (0,0) node[align=center, text height=5pt, text height=5pt,   below=1.2pt] {$a_2$};
    \draw[color=black] (1,0) node[align=center, text height=5pt, text height=5pt,   below=1.2pt] {$a_1$};
    \draw[color=black] (5,0) node[align=center, text height=5pt, text height=5pt,   below=1.2pt] {$b_1$};
    \draw[color=black] (2,0) node[align=center, text height=5pt, text height=5pt,   below=1.2pt] {$b_2$};
    \draw[color=black] (3,0) node[align=center, text height=5pt, text height=5pt,   below=1.2pt] {$c_2$};
    \draw[color=black] (4,0) node[align=center, text height=5pt, text height=5pt,   below=1.2pt] {$d_2$};
    \end{tikzpicture}
    
    \end{figure}
    
    \vspace{-.75cm}
    \begin{equation}
    \label{eq:K23_no3+2edges}
        y_{a_2a_1,a_2b_1} + \sum\limits_{c \notin\set{a_1,a_2,b_1}}y_{a_1b_2,a_1c} \leq 2 \qquad \forallpwd a_1,b_1,a_2,b_2 \in \nint
    \end{equation}
    \end{mdframed}

 Note that constraint~\eqref{eq:K23_no3+2edges} implies the previous constraint~\eqref{eq:K23_no4edges}.

\begin{mdframed}
    \begin{figure}[H]
    \center
    \raisebox{2em}{
    \begin{tikzpicture}[scale=0.7]
    \draw[ultra thick, BrickRed](0,0) to (1,0.6);
    \draw[ultra thick, BrickRed](0,0) to (1,-0.6);
    \draw[ultra thick, NavyBlue](1.4,0) to (1,0.6);
    \draw[ultra thick, NavyBlue](1.4,0) to (1,-0.6);
    \draw[ultra thick, Green](2,0) to (1,0.6);
    \draw[ultra thick, Green](2,0) to (1,-0.6);
 \foreach \x/\y in {0/0, 1/0.6, 1/-0.6, 1.4/0, 2/0}{\draw (\x,\y) node[vertex_u]{};  }         
    \end{tikzpicture}}
    \hspace{1.5cm}
    \begin{tikzpicture}[scale=0.8]
    \draw[ultra thick, BrickRed] (0,2) to (1,2.5);
    \draw[ultra thick, BrickRed] (0,1) to (1,2.5);
    \draw[ultra thick, NavyBlue] (0,2) to (1,1.85);
    \draw[ultra thick, NavyBlue] (0,1) to (1,1.85);
    \draw[ultra thick, Green] (0,2) to (1,1.15);
    \draw[ultra thick, Green] (0,1) to (1,1.15);
    \draw[thick] (0,2) to (1,0.5);
    \draw[thick] (0,1) to (1,0.5);
    \draw[color=black] (0,2) node[align=center,   left=1.2pt] {$a_1$};
    \draw[color=black] (0,1) node[align=center,   left=1.2pt] {$b_1$};
    \draw[color=black] (1,2.5) node[align=center,   right=1.2pt] {$a_2$};
    \draw[color=black] (1,1.85) node[align=center,   right=1.2pt] {$b_2$};
    \draw[color=black] (1,1.15) node[align=center,   right=1.2pt] {$c_2$};
    \draw[color=black] (1,0.5) node[align=center,   right=1.2pt] {$d_2$};
 \foreach \x/\y in {0/1, 0/2, 1/0.5, 1/1.85, 1/1.15, 1/2.5}{\draw (\x,\y) node[vertex_u]{};  }         
    \end{tikzpicture}
    \hspace{1.5cm}
    \begin{tikzpicture}[scale=.6]
    \foreach \x in {0,1,2,3,4,5}{\draw (\x,0) node[vertex_u]{};}      
    \draw[ultra thick, BrickRed](0,0) to[bend left=50] (3,0);
    \draw[ultra thick, BrickRed](0,0) to[bend left=60] (4,0);
    \draw[ultra thick, NavyBlue](1,0) to[bend left=40] (3,0);
    \draw[ultra thick, NavyBlue](1,0) to[bend left=50] (4,0);
    \draw[ultra thick, Green](2,0) to[bend left=25] (3,0);
    \draw[ultra thick, Green](2,0) to[bend left=40] (4,0);
    \draw[thick] (3,0) to[bend left=50] (5,0);
    \draw[thick] (4,0) to[bend left=40] (5,0);
    \draw[color=black] (0,0) node[align=center, text height=5pt,   below=1.2pt] {$a_2$};
    \draw[color=black] (1,0) node[align=center, text height=5pt,   below=1.2pt] {$b_2$};
    \draw[color=black] (2,0) node[align=center, text height=5pt,   below=1.2pt] {$c_2$};
    \draw[color=black] (3,0) node[align=center, text height=5pt,   below=1.2pt] {$a_1$};
    \draw[color=black] (4,0) node[align=center, text height=5pt,   below=1.2pt] {$b_1$};
    \draw[color=black] (5,0) node[align=center, text height=5pt,   below=1.2pt] {$d_2$};
    \end{tikzpicture}
    
    \end{figure}
    
    \vspace{-.5cm}
    \begin{equation*} \label{eq:K23_no3edgeswith2a}
        \sum\limits_{e \notin\set{a_1,b_1}}y_{ea_1,eb_1} \leq 2 \qquad \foralld a_1, b_1 \in \nint
    \end{equation*}
    
    \end{mdframed}

\begin{mdframed}
    \begin{figure}[H]
    \center
    \raisebox{2em}{
    \begin{tikzpicture}[scale=0.7]
    \draw[ultra thick, BrickRed](0,0) to (1,0.5);
    \draw[ultra thick, BrickRed](0,0) to (1,-0.5);
    \draw[ultra thick, NavyBlue](2,0) to (1,0.5);
    \draw[ultra thick, NavyBlue](2,0) to (1,-0.5);
    \draw[ultra thick, Green](1,0.5) to (1.75,0.9);
    \draw[ultra thick, Green](1,0.5) to (0.25,0.9);
 \foreach \x/\y in {0/0, 1/0.5, 1/-0.5, 1.75/0.9, 0.25/0.9, 2/0}{\draw (\x,\y) node[vertex_u]{};  }         
    \end{tikzpicture}}
    \hspace{1.5cm}
    \begin{tikzpicture}[scale=0.8]
    \draw[ultra thick, BrickRed] (0,2) to (1,2.5);
    \draw[ultra thick, BrickRed] (0,1) to (1,2.5);
    \draw[ultra thick, NavyBlue] (0,2) to (1,1.85);
    \draw[ultra thick, NavyBlue] (0,1) to (1,1.85);
    \draw[ultra thick, Green] (0,2) to (1,1.15);
    \draw[thick] (0,1) to (1,1.15);
    \draw[ultra thick, Green] (0,2) to (1,0.5);
    \draw[thick] (0,1) to (1,0.5);
    \draw[color=black] (0,2) node[align=center,   left=1.2pt] {$a_1$};
    \draw[color=black] (0,1) node[align=center,   left=1.2pt] {$b_1$};
    \draw[color=black] (1,2.5) node[align=center,   right=1.2pt] {$a_2$};
    \draw[color=black] (1,1.85) node[align=center,   right=1.2pt] {$b_2$};
    \draw[color=black] (1,1.15) node[align=center,   right=1.2pt] {$c_2$};
    \draw[color=black] (1,0.5) node[align=center,   right=1.2pt] {$d_2$};
 \foreach \x/\y in {0/1, 0/2, 1/0.5, 1/1.85, 1/1.15, 1/2.5}{\draw (\x,\y) node[vertex_u]{};  }         
    \end{tikzpicture}
    \hspace{1.5cm}
    \begin{tikzpicture}[scale=.6]
    \draw[ultra thick, BrickRed](0,0) to[bend left=50] (2,0);
    \draw[ultra thick, BrickRed](0,0) to[bend left=60] (5,0);
    \draw[ultra thick, NavyBlue](1,0) to[bend left=25] (2,0);
    \draw[ultra thick, NavyBlue](1,0) to[bend left=50] (5,0);
    \draw[ultra thick, Green](2,0) to[bend left=25] (3,0);
    \draw[ultra thick, Green](2,0) to[bend left=50] (4,0);
    \draw[thick] (3,0) to[bend left=50] (5,0);
    \draw[thick] (4,0) to[bend left=40] (5,0);
    \draw[color=black] (0,0) node[align=center, text height=5pt,   below=1.2pt] {$a_2$};
    \draw[color=black] (1,0) node[align=center, text height=5pt,   below=1.2pt] {$b_2$};
    \draw[color=black] (2,0) node[align=center, text height=5pt,   below=1.2pt] {$a_1$};
    \draw[color=black] (3,0) node[align=center, text height=5pt,   below=1.2pt] {$c_2$};
    \draw[color=black] (4,0) node[align=center, text height=5pt,   below=1.2pt] {$d_2$};
    \draw[color=black] (5,0) node[align=center, text height=5pt,   below=1.2pt] {$b_1$};
    \foreach \x in {0,1,2,3,4,5}{\draw (\x,0) node[vertex_u]{};}        
    \end{tikzpicture}

    \end{figure}

    \vspace{-.8cm}
    \begin{equation*}
        y_{a_2a_1, a_2b_1} + y_{b_2a_1, b_2b_1} + y_{a_1c_2,a_1d_2} \leq 2 \qquad \forallpwd a_2,b_2, a_1,b_1,c_2,d_2 \in \nint
    \end{equation*}
    \end{mdframed}

\newpage
\subsection{Complete bipartite graph $\Kb{3}{3}$}

\begin{mdframed} 
    \begin{figure}[H]
    \center
    \raisebox{2em}{
    \begin{tikzpicture}[scale=0.7]
    \draw[ultra thick, BrickRed](0,0) to (1,0.75);
    \draw[ultra thick, BrickRed](0,0) to (1,0);
    \draw[ultra thick, BrickRed](0,0) to (1,-0.75);
    \draw[ultra thick, NavyBlue](2,0) to (1,0.75);
    \draw[ultra thick, NavyBlue](2,0) to (1,0);
    \draw[ultra thick, NavyBlue](2,0) to (1,-0.75);
 \foreach \x/\y in {0/0, 1/0, 1/0.75, 1/-0.75, 2/0}{\draw (\x,\y) node[vertex_u]{};  }         
    \end{tikzpicture}}
    \hspace{1.5cm}
    \begin{tikzpicture}[scale=1]
    \draw[ultra thick, BrickRed] (0,2) to (1,2);
    \draw[ultra thick, BrickRed] (0,2) to (1,1);
    \draw[ultra thick, BrickRed] (0,2) to (1,0);
    \draw[ultra thick, NavyBlue] (0,1) to (1,2);
    \draw[ultra thick, NavyBlue] (0,1) to (1,1);
    \draw[ultra thick, NavyBlue] (0,1) to (1,0);
    \draw[thick] (0,0) to (1,2);
    \draw[thick] (0,0) to (1,1);
    \draw[thick] (0,0) to (1,0);
    \draw[color=black] (0,2) node[align=center,   left=1.2pt] {$a_1$};
    \draw[color=black] (0,1) node[align=center,   left=1.2pt] {$b_1$};
    \draw[color=black] (0,0) node[align=center,   left=1.2pt] {$c_1$};
    \draw[color=black] (1,2) node[align=center,   right=1.2pt] {$a_2$};
    \draw[color=black] (1,1) node[align=center,   right=1.2pt] {$b_2$};
    \draw[color=black] (1,0) node[align=center,   right=1.2pt] {$c_2$};
 \foreach \x/\y in {0/0, 0/1, 0/2, 1/0, 1/1, 1/2}{\draw (\x,\y) node[vertex_u]{};  }     
    \end{tikzpicture}
    \hspace{1.5cm}
    \begin{tikzpicture}[scale=1]
    \foreach \x in {0,1,2,3,4,5}{\draw (\x,0) node[vertex_u]{};}      
    \draw[ultra thick, BrickRed](0,0) to[bend left=40] (2,0);
    \draw[ultra thick, BrickRed](0,0) to[bend left=50] (3,0);
    \draw[ultra thick, BrickRed](0,0) to[bend left=60] (4,0);
    \draw[ultra thick, NavyBlue](1,0) to[bend left=25] (2,0);
    \draw[ultra thick, NavyBlue](1,0) to[bend left=40] (3,0);
    \draw[ultra thick, NavyBlue](1,0) to[bend left=50] (4,0);
    \draw[thick] (2,0) to[bend left=50] (5,0);
    \draw[thick] (3,0) to[bend left=40] (5,0);
    \draw[thick] (4,0) to[bend left=25] (5,0);
    \draw[color=black] (0,0) node[align=center, text height=5pt,   below=1.2pt] {$a_1$};
    \draw[color=black] (1,0) node[align=center, text height=5pt,   below=1.2pt] {$b_1$};
    \draw[color=black] (2,0) node[align=center, text height=5pt,   below=1.2pt] {$a_2$};
    \draw[color=black] (3,0) node[align=center, text height=5pt,   below=1.2pt] {$b_2$};
    \draw[color=black] (4,0) node[align=center, text height=5pt,   below=1.2pt] {$c_2$};
    \draw[color=black] (5,0) node[align=center, text height=5pt,   below=1.2pt] {$c_1$};
    \end{tikzpicture}
    
    \end{figure}
    
    \vspace{-.75cm}
    \begin{align*}
    y_{a_1a_2,a_1b_2} + y_{a_1a_2,a_1c_2} + y_{a_1b_2,a_1c_2} + y_{b_1a_2,b_1b_2} + y_{b_1a_2,b_1c_2} + y_{b_1b_2,b_1c_2} \leq 4& \\  \forallpwd &a_1,b_1,a_2,b_2,c_2  \in \nint
    \end{align*}
    \end{mdframed} 
    
\begin{mdframed}
    \begin{figure}[H]
    \center
    \raisebox{2em}{
    \begin{tikzpicture}[scale=0.7]
    \draw[ultra thick, BrickRed](0,0) to (1,0.75);
    \draw[ultra thick, BrickRed](0,0) to (1,0);
    \draw[ultra thick, BrickRed](0,0) to (1,-0.75);
    \draw[ultra thick, NavyBlue](1,0.75) to (2,1.25);
    \draw[ultra thick, NavyBlue](1,0.75) to (2,0);
    \draw[ultra thick, Green](2,0) to (1,0);
    \draw[ultra thick, Green](2,0) to (1,-.75);
 \foreach \x/\y in {0/0, 1/0, 1/0.75, 1/-0.75, 2/0, 2/1.25}{\draw (\x,\y) node[vertex_u]{};  }         
    \end{tikzpicture}}
    \hspace{1.5cm}
    \begin{tikzpicture}[scale=1]
    \draw[ultra thick, BrickRed] (0,2) to (1,2);
    \draw[ultra thick, BrickRed] (0,2) to (1,1);
    \draw[ultra thick, BrickRed] (0,2) to (1,0);
    \draw[ultra thick, NavyBlue] (0,1) to (1,2);
    \draw[ultra thick, Green] (0,1) to (1,1);
    \draw[ultra thick, Green] (0,1) to (1,0);
    \draw[ultra thick, NavyBlue] (0,0) to (1,2);
    \draw[thick] (0,0) to (1,1);
    \draw[thick] (0,0) to (1,0);
    \draw[color=black] (0,2) node[align=center,   left=1.2pt] {$a_1$};
    \draw[color=black] (0,1) node[align=center,   left=1.2pt] {$b_1$};
    \draw[color=black] (0,0) node[align=center,   left=1.2pt] {$c_1$};
    \draw[color=black] (1,2) node[align=center,   right=1.2pt] {$a_2$};
    \draw[color=black] (1,1) node[align=center,   right=1.2pt] {$b_2$};
    \draw[color=black] (1,0) node[align=center,   right=1.2pt] {$c_2$};
 \foreach \x/\y in {0/0, 0/1, 0/2, 1/0, 1/1, 1/2}{\draw (\x,\y) node[vertex_u]{};  }   
 \end{tikzpicture}
    \hspace{1.5cm}
    \begin{tikzpicture}[scale=1]
    \draw[ultra thick, BrickRed](0,0) to[bend left=25] (1,0);
    \draw[ultra thick, BrickRed](0,0) to[bend left=40] (3,0);
    \draw[ultra thick, BrickRed](0,0) to[bend left=50] (4,0);
    \draw[ultra thick, NavyBlue](1,0) to[bend left=25] (2,0);
    \draw[ultra thick, NavyBlue](1,0) to[bend left=50] (5,0);
    \draw[ultra thick, Green] (2,0) to[bend left=25] (3,0);
    \draw[ultra thick, Green] (2,0) to[bend left=40] (4,0);
    \draw[thick] (3,0) to[bend left=40] (5,0);
    \draw[thick] (4,0) to[bend left=25] (5,0);
    \draw[color=black] (0,0) node[align=center, text height=5pt,   below=1.2pt] {$a_1$};
    \draw[color=black] (1,0) node[align=center, text height=5pt,   below=1.2pt] {$a_2$};
    \draw[color=black] (2,0) node[align=center, text height=5pt,   below=1.2pt] {$b_1$};
    \draw[color=black] (3,0) node[align=center, text height=5pt,   below=1.2pt] {$b_2$};
    \draw[color=black] (4,0) node[align=center, text height=5pt,   below=1.2pt] {$c_2$};
    \draw[color=black] (5,0) node[align=center, text height=5pt,   below=1.2pt] {$c_1$};
    \foreach \x in {0,1,2,3,4,5}{\draw (\x,0) node[vertex_u]{};}     
    \end{tikzpicture}
    
    \end{figure}
     \vspace{-.8cm}
    \begin{align*}
    y_{a_1a_2,a_1b_2} + y_{a_1a_2,a_1c_2} + y_{a_1b_2,a_1c_2} + y_{a_2b_1,a_2c_1} &+  y_{b_1b_2,b_1c_2} \leq 4 \\
      &\forallpwd a_1,b_1,c_1,a_2,b_2,c_2  \in \nint
    \end{align*}
\end{mdframed}
\begin{mdframed}
    \begin{figure}[H]
    \center
    \raisebox{2em}{
    \begin{tikzpicture}[scale=0.7]
    \draw[ultra thick, BrickRed](0,0) to (1,0.75);
    \draw[ultra thick, BrickRed](0,0) to (1,0);
    \draw[ultra thick, BrickRed](0,0) to (1,-0.75);
    \draw[ultra thick, NavyBlue](1,0.75) to (2,1.25);
    \draw[ultra thick, NavyBlue](1,0.75) to (2,0);
    \draw[ultra thick, Green](2,0) to (1,0);
    \draw[ultra thick, Green](2,1.25) to (1,0);
 \foreach \x/\y in {0/0, 1/0, 1/0.75, 1/-0.75, 2/0, 2/1.25}{\draw (\x,\y) node[vertex_u]{};  }     
    \end{tikzpicture}}
    \hspace{1.5cm}
    \begin{tikzpicture}[scale=1]
    \draw[ultra thick, BrickRed] (0,2) to (1,2);
    \draw[ultra thick, BrickRed] (0,2) to (1,1);
    \draw[ultra thick, BrickRed] (0,2) to (1,0);
    \draw[ultra thick, NavyBlue] (0,1) to (1,2);
    \draw[ultra thick, Green] (0,1) to (1,1);
    \draw[thick] (0,1) to (1,0);
    \draw[ultra thick, NavyBlue] (0,0) to (1,2);
    \draw[ultra thick, Green] (0,0) to (1,1);
    \draw[thick] (0,0) to (1,0);
    \draw[color=black] (0,2) node[align=center,   left=1.2pt] {$a_1$};
    \draw[color=black] (0,1) node[align=center,   left=1.2pt] {$b_1$};
    \draw[color=black] (0,0) node[align=center,   left=1.2pt] {$c_1$};
    \draw[color=black] (1,2) node[align=center,   right=1.2pt] {$a_2$};
    \draw[color=black] (1,1) node[align=center,   right=1.2pt] {$b_2$};
    \draw[color=black] (1,0) node[align=center,   right=1.2pt] {$c_2$};
 \foreach \x/\y in {0/0, 0/1, 0/2, 1/0, 1/1, 1/2}{\draw (\x,\y) node[vertex_u]{};  }        
    \end{tikzpicture}
    \hspace{1.5cm}
    \begin{tikzpicture}[scale=1]
    \draw[ultra thick, BrickRed](0,0) to[bend left=25] (1,0);
    \draw[ultra thick, BrickRed](0,0) to[bend left=45] (5,0);
    \draw[ultra thick, BrickRed](0,0) to[bend left=40] (2,0);
    \draw[ultra thick, NavyBlue](1,0) to[bend left=40] (3,0);
    \draw[ultra thick, NavyBlue](1,0) to[bend left=50] (4,0);
    \draw[ultra thick, Green] (2,0) to[bend left=25] (3,0);
    \draw[ultra thick, Green] (2,0) to[bend left=40] (4,0);
    \draw[thick] (3,0) to[bend left=40] (5,0);
    \draw[thick] (4,0) to[bend left=25] (5,0);
    \draw[color=black] (0,0) node[align=center, text height=5pt,   below=1.2pt] {$a_1$};
    \draw[color=black] (1,0) node[align=center, text height=5pt,   below=1.2pt] {$a_2$};
    \draw[color=black] (2,0) node[align=center, text height=5pt,   below=1.2pt] {$b_2$};
    \draw[color=black] (3,0) node[align=center, text height=5pt,   below=1.2pt] {$b_1$};
    \draw[color=black] (4,0) node[align=center, text height=5pt,   below=1.2pt] {$c_1$};
    \draw[color=black] (5,0) node[align=center, text height=5pt,   below=1.2pt] {$c_2$};
    \foreach \x in {0,1,2,3,4,5}{\draw (\x,0) node[vertex_u]{};}         
    \end{tikzpicture}

    \end{figure}    
    
    \vspace{-.8cm}
    \begin{align*}
    y_{a_1a_2,a_1b_2} + y_{a_1a_2,a_1c_2} + y_{a_1b_2,a_1c_2} + y_{a_2b_1,a_2c_1} &+  y_{b_2b_1,b_2c_1} \leq 4 \\
      &\forallpwd a_1,b_1,c_1,a_2,b_2,c_2  \in \nint
    \end{align*}
    \end{mdframed}

\end{document}